\documentclass[10pt]{amsart}
\usepackage{amssymb}
\usepackage{amsmath,amssymb,amsthm,graphicx}
\usepackage{comment}
\usepackage{enumerate}
\usepackage[toc,page,title,titletoc,header]{appendix}
\usepackage[colorlinks=true]{hyperref}
\usepackage{xcolor}
\usepackage{cases}
\usepackage{marginnote}
\newtheorem{theorem}{Theorem}
\newtheorem{meta-thm}[theorem]{Meta-Theorem}
\newtheorem{lemma}[theorem]{Lemma}

\newtheorem{proposition}[theorem]{Proposition}

\newtheorem{remark}[theorem]{Remark}
\newtheorem{definition}[theorem]{Definition}

\newtheorem{assumption}[theorem]{Assumption}
\numberwithin{equation}{section}
\newcommand{\ZZ}{\mathbb{Z}}

\newcommand{\real}{\mathbb{R}}

\newcommand{\torus}{\mathbb{T}}
\def\ep{\epsilon}

\def\F{\mathcal{F}}
\def\Lip{\text{Lip}}

\title[The existence of solutions for nonlinear elliptic equations]
{The existence of solutions for nonlinear elliptic equations: Simple proofs
 and extensions of a paper by Y. Shi }

\author[X.Xu]{Xiaodan Xu}
\address{
	School of Mathematical Sciences, University of Jinan,
	Jinan, 250022, P.R.China.
}
\email{xuxiaodanmath@163.com}

\author[R.de la Llave]{Rafael de la Llave}
\address{School of Mathematics, Georgia Inst. of Technology,
	Atlanta GA, 30332, USA}
\email{rafael.delallave@math.gatech.edu}

\author[F. Wang]{Fenfen Wang}
\address{School of Mathematical Sciences and Laurent Mathematics Center, Sichuan Normal University, Chengdu 610066, P.R.China.}
\email{ffenwang@hotmail.com}

\date{\today}

\thanks{The work of R.L. was supported by NSF grant DMS-1800241.
X.X was supported by CSC by the National Natural Science
Foundation of China (Grant Nos. 11971261; 11571201). X.X. thanks  the
School of Mathematics of Georgia Inst. of Technology for hospitality
during the academic year 2019-2020 and Fall 2020. F.W. was supported by the Fundamental Research Funds of Sichuan Normal University (KY. 20200921).}

\begin{document}
\maketitle
\begin{abstract}

The paper \cite{Shi} uses the Craig-Wayne-Bourgain
method to construct
solutions of an elliptic problem involving parameters.
The results of \cite{Shi} include  regularity assumptions on the perturbation
and involve excluding parameters. The paper \cite{Shi} also
constructs response solutions to a quasi-periodically perturbed (ill-posed
evolution) problem.

In this paper, we use several classical methods (freezing of coefficients, alternative methods for nonlinear elliptic equations)
to extend the results of \cite{Shi}. We weaken the regularity
assumptions on the perturbation and we describe the phenomena that
happens for all parameters.
In the ill-posed problem, we use a recently developed time-dependent
center manifold theorem which allows to reduce the problem
to a finite-dimensional ODE with quasi-periodic dependence on
time.  The bounded and sufficiently small solutions of these ODE
give solutions of the ill-posed PDE.
\end{abstract}



\section{Introduction}
\subsection{Previous results}
The recent paper \cite{Shi}
considers the problem
\begin{equation}
-\Delta u-\mathfrak{m}u+\epsilon f(x,u)=0,\,\,x\in\mathcal{D}:=\mathbb{R}^{d}/ (2\pi\beta_{i}\mathbb{Z})^{d},\,\,d\in\mathbb{Z}_{+},
\label{eq0}
\end{equation}
where
$\mathfrak{m}>0$ (as we will see later, the case $\mathfrak{m} \le 0$ is
easy),
$\epsilon\geq0$
and
$\beta=(\beta_{1},\cdots,\beta_{d})\in[1/2,1]^{d}.$
The unknown function is $u: \mathcal{D}\rightarrow\mathbb{R},$
and data is the nonlinearity
$f: \mathcal{D}\times\mathbb{R}\rightarrow\mathbb{R}$,
which in \cite{Shi} is assumed to be a polynomial in $u$ with  coefficients that
are trigonometric polynomials in $x$.

The paper \cite{Shi} uses the Craig-Wayne-Bourgain (CWB) method
\cite{CraigW93,CraigW94, Bourgain94, Craig00, Bourgain05}
to prove existence of analytic solutions of \eqref{eq0}
when $(\beta_{1}^{-1},\cdots,
\beta_{d}^{-1})\in [1,2]^{d}$ lies in an appropriate set
whose measure is estimated.

The paper \cite{Shi} also considers the formal  ``evolution'' problem
\begin{equation}\label{evolution}
-u_{tt}-\Delta u-\frak{m}u+\epsilon f(t,x,u)=0, \,\,x\in\mathcal{D},
\end{equation}
where $\frak{m}>0, \epsilon\geq0, \beta=(\beta_{1},\beta_{2},\cdots,\beta_{d})\in[1/2,1]^{d}.$
$f: \mathbb{R}\times \mathcal{D}\times\mathbb{R}\rightarrow\mathbb{R}$ is quasi-periodic with respect to time $t$ with frequency vector $\omega\in\mathbb{R}^{b},$ $b\in\mathbb{Z}_{+}.$
The paper \cite{Shi} produces response solutions (i.e., quasi-periodic solutions with the same
frequency as the forcing).

Note that the differential operator in \eqref{evolution}
is also an elliptic operator, so considering it as an
evolution equation leads to an ill-posed problem. Nevertheless,
even if one cannot produce solutions for all
initial conditions, it is possible to obtain interesting solutions.
Indeed, the consideration of elliptic problems in cylindrical
domains as ``evolution'' problems has been considered in
several papers \cite{KirchgassnerS79, Mielke91} and, more recently,
\cite{dlLR09, PolacikV17, PolacikV20,  CdlLR20}.

\subsection{The results in this paper}

In this paper, we revisit and extend the results above using
some classical methods (freezing of coefficients, alternative method)
for the problem \eqref{eq0} or some more modern methods
(reduction to center manifolds for ill-posed equations  \cite{Mielke91,
dlLR09, CdlLR20})
for the problem \eqref{evolution}.

An outline of the main ideas is as follows: as for the treatment of \eqref{eq0} we distinguish
whether
spectrum of $-\Delta -\frak{m}$ is away from zero  (we will
call this cases \emph{non-resonant})  or whether the spectrum of
 $-\Delta -\frak{m}$  contains zero (we call these cases
\emph{resonant}).

\subsubsection{Treatment of \eqref{eq0} when $-\Delta -\frak{m}$ is invertible: freezing of coefficients}
When the spectrum of    $-\Delta -\frak{m}$ is away from zero
\footnote{In general, the spectrum depends on the space one
is considering the operator acting. However, for elliptic operators in bounded
domains, the spectrum is largely independent of the space.  Later we
will specify which spaces we are considering.},
 we  transform
\eqref{eq0} into a fixed point problem in an appropriate Banach
space. This allows us to remove the assumption in \cite{Shi}
that the nonlinearity is polynomial in $x,u$ and it also
allows us to deal with nonlinearities that involve  the
derivatives of $u$ up to order $2$.
That is, we allow nonlinearities $f(x,u, Du(x), D^2u(x))$ and even more general functions denoted by $\F[u]$. This extra generality includes
several interesting cases, that have attracted
attention in recent times such as fractional derivatives,
$(-\Delta)^{\alpha}u $, the Kirkhoff terms
$ ( \int_{\torus^d} |\nabla u|^2 )\Delta u $ or the water wave
terms $(-\Delta)^{1/2} \tanh( (-\Delta)^{1/2}) u $.
We will just require that the functional $\F$
is  Lipschitz  mapping from a space of differentiable
functions to another space of differentiable functions
(with two derivatives less). As it is
well known from classical potential theory, the scale of spaces has to
be carefully chosen so that the gain of regularity obtained
by applying $(-\Delta -\frak{m})^{-1}$ compensates the loss of
regularity incurred by $\F$.

\begin{remark}
For the experts in classical
elliptic regularity theory
\cite{Agmon,AgmonDN59,AgmonDN64} \cite[Chapter 15]{Tay11b}
we anticipate
 that the method is very similar to the classic ``freezing of coefficients''
but that in our case, we do not need to localize the problem, so
that we do not need to use commutator estimates, which makes it
possible to obtain analytic results
for analytic $f$.

The spaces we work with are chosen so that they can be
analytic functions for some values of the parameters and finite-differentiable
functions for other parameter values.
\end{remark}

The results on analytic (and finite-differentiable regularity) depend crucially  on choosing a remarkable
family (indexed by two parameters) of function spaces where
to formulate the functional
analysis  problem.

This family of spaces has been used
in the past, \cite{CallejaCL13, CCCdlL17,WangL20}. These spaces enjoy many remarkable
properties (presented here in Appendix~\ref{sec:properties})
including that they are Banach algebras from some ranges of
the parameters.
In this paper we prove Lemma~\ref{Banachalgebra}, which improves the range of
parameters for the Banach algebra properties established in \cite{CallejaCL13, CCCdlL17,WangL20}.
This immediately leads to improvements in the
range of parameters in the above references.  In the notation
of the above references, the assumption $r > d$  in the above
papers can be
weakened to $r > d/2$ using Lemma~\ref{Banachalgebra} in the present paper.

\subsubsection{Treatment of \eqref{eq0} when $-\Delta - \frak{m} $ is not invertible: bifurcation theory}

When the spectrum of $-\Delta -\frak{m}$ contains zero
(a problem not considered in \cite{Shi}), we note that zero is an eigenvalue of
finite multiplicity so that one can apply the
classical Cesari alternative method
\cite{Cesari} (also called Lyapunov-Schmidt reduction
\cite{Kielhofer12}).  In our case, there are some unusual properties
such as the kernel having large dimension and the presence of symmetries,
so that the calculations involve several algebraic surprises. The
algebraic difficulty increases with the dimension, so we present a complete
example
in  dimesions 1 and 2.

 With some appropriate conditions on the nonlinearity, we can indeed
obtain smooth branches of solutions.
\medskip

Putting together the two results, we can obtain results for all
the choices of $\nu, \frak{m}$ provided  that some explicit
non-degeneracy conditions on the nonlinearity hold.

\subsection{Treatment of \eqref{evolution}: fixed point methods and reduction principles}

As for the equation \eqref{evolution}, it is natural to consider
the equation acting on a space of quasi-periodic functions. The
spectrum of the linearized system always contains semi-lines
 when the
frequencies $\omega$ have dimension $2$ or more.
An elementary result along the lines of the previous result is obtained by
assumming that the spectrum does not contain zero which happens for
some $\beta, \frak{m}, \omega$.

A more sophisticated method to study \eqref{evolution}, which applies
to all $\beta, \frak{m}$ is to observe that we can apply the
time-dependent center manifold theorem introduced in \cite{CdlLR20} to
establish the existence of the time-dependent invariant manifold for the
evolution equation \eqref{evolution}. Using this center manifold,
we can reduce the original problem to a finite-dimensional quasi-periodic ODE problem.  The fixed point methods presented here allows nonlinearities
that loose two derivatives.
The reduction principle of \cite{CdlLR20}  allows nonlinearities
that loose $2-\kappa$ derivatives. In this paper, we will
present a detailed proof of the simpler case when there
is only one derivative.

The study of solutions for  quasi-periodic equation in finite
dimensions    is well developed and
there are a large variety of techniques (which  cannot  be even reviewed) to
produce interesting solutions \cite{Mitropolsky, Melnikov}. These well
studied solutions include indeed response solutions, but also
subharmonic response solutions and many others.

If we find solutions of the problem which remain in a
small enough neighborhood of the origin,
they will become solutions to the original problem
(remember that the center manifold is
locally invariant). Therefore, we can produce solutions of the original
problem, just by producing solutions of a finite-dimensional problem.
Such procedures are often called \emph{reduction principles}.

It is interesting to mention that the reduction to finite dimensions
in \cite{CdlLR20} does not require any assumptions on the perturbing
frequencies.  Of course, the analysis of the resulting
finite-dimensional system using KAM theory may require that the
frequency satisfies some number theoretic properties. Other methods
may have other assumptions, but we will not detail them here.

\subsection{Organization of this paper}

This paper is organized as follows:
In Section \ref{mainidea}, we present the main idea of proving the existence of solutions through the several classical methods we mentioned above.
In Section \ref{space}, we
introduce  a convenient two parameter family
of  function spaces $H^{\rho, r}.$
When $\rho > 0$, the space $H^{\rho, r}$ consists of
analytic functions, but $H^{0,r}$ is the standard Sobolev space.
In Section~
\ref{nonresonant}-\ref{center}, we give our several main results which are the simple proof of the results in \cite{Shi} and the extensions.
Precisely, in Section \ref{nonresonant}, we study the case when the spectrum of the operator $-\Delta-\mathfrak{m}$ of equation \eqref{eq0} is non-resonant.
In Section \ref{resonant}, we introduce the Cesari alternative method to deal with the case that the spectrum of $-\Delta-\mathfrak{m}$ is resonant.
For the ill-posed evolution equation \eqref{evolution}, when the spectrum of $-\partial_{tt}-\Delta-\mathfrak{m}$ is non-resonant, we introduce our results in Section \ref{illposedcase} and for the reosnant case, we introduce the center manifold theorem to solve \eqref{evolution} in Section \ref{center}.

\section{A preview of the results}\label{mainidea}
In this section, we describe formally  the methods we will use,
ignoring for the moment  questions of spaces, domains, etc.
These will be taken care later. The precise definitions will be motivated
by the desire to make the formal manipulations go through.

\subsection{Elliptic theory away from resonances}

We consider the variable $x$ on $\mathbb{T}^{d}:=\mathbb{R}^{d}/(2\pi\mathbb{Z})^{d}.$
To rewrite \eqref{eq0} in a more
convenient way,
\def\LL{{ \mathcal{L}_{\nu,\mathfrak{m} } }}
we denote by $\LL$ the linear operator
\begin{equation}
\LL =\sum_{i=1}^{d}\nu_{i}^{2}\frac{\partial^{2}}{\partial x_{i}^{2}}+\mathfrak{m},
\end{equation}
where $\nu=(\nu_{1},\cdots,\nu_{d})=(\beta_{1}^{-1},\cdots,\beta_{d}^{-1})\in[1,2]^{d}.$ Moreover, we allow that the nonlinearity $f$ also depends on $Du(x), D^{2}u(x).$
For convenience, we  represent the
nonlinearity by  $\F(u)(x):= f(x,u, Du, D^2u)$, then it suffices to verify the abstract
hypothesis for $\F$.

Then, the equation \eqref{eq0} becomes
\begin{equation}
(\LL u)(x)=\epsilon \F(u)(x),\,\, x\in \mathbb{T}^d.
\label{operator}
\end{equation}
We notice that $\LL$ is a diagonal operator in the Fourier basis. Precisely,
\[
\LL(\exp\{\mathrm{i}kx\})=\Upsilon_{k}\exp\{\mathrm{i}kx\}
\]
with $\Upsilon_{k}=\sum_{i=1}^{d}-\nu_{i}^{2}k_{i}^{2}+\mathfrak{m}, k=(k_1,\cdots,k_d)\in\mathbb{Z}^{d}.$

\begin{remark} \label{flipping}
Since  flipping the signs of components of $k$ does not change
the eigenvalue of the operator $\LL$ (but it changes the eigenvector when some of the components
of $k$ are not zero),  then
the eigenvalues have always multiplicity at least $2^{\eta(k)}$
where $\eta(k)$ is the number of components of $k$ which
are not zero. Of course, we have $0 \le \eta(k) \le d$ and, for $\eta(k)$  outside of the coordinate hyperplanes $\eta(k) =
 d$, we need $\nu_i^2$ have some
 rational relations.

If the $\nu_i^2$ have some
rational relations, the multiplicity could be higher\footnote{If an eigenvalue
has multiplicity bigger than $2^{\eta(k)}$, the $\nu_i$ should satisfy
some linear relations. Therefore, except for a set of measure zero of
the $\nu_i$, all the eigenvalues will have multiplicity exactly
$2^{\eta(k)}$. Note that $\eta(k)\leq d,$ so in dimensional $1$ the kernel will have dimension $2$.}
but this happens in a set of measure zero of $\nu$.

For simplicity of the discussion, we will assume that
the $\nu$ we consider have no rational relation. We will furthermore
assume that $\eta(k) = d$.
See  Assumption~\ref{simplicity}.
This assumption could be avoided
with longer explicit calculations.
\end{remark}

In reasonable spaces, for which exponentials
 will be a basis (for example in the spaces presented in Section~\ref{space})
 $\LL$ will be
a self-adjoint operator and  the $\Upsilon_{k}$ will be
its spectrum. Notice that this spectrum is a discrete set going to infinity
and that the eigenvalues have finite multiplicity. (Again, we recall
that, for the present operators, the spectrum is largely independent
of the space we consider it.  We will of course, make the spaces
explicit later since the choice of spaces plays a big role in the
treatment of nonlinear terms).

Since we will use functional analysis,  we will find it convenient to consider that the right hand side of
\eqref{operator} is written as $\F[u]$ and we will think of
$\F$ as a mapping that maps a space of functions
with a certain number of derivatives to another spaces of functions
(which have possibly less derivatives).

 We will prove our
results under abstract assumptions on the operator $\F$. Afterwards will show
that if $\F[u]$ is given by
\begin{equation} \label{concreteF}
\F[u](x) = f(x, u(x), Du(x), D^2 u(x) )
\end{equation}
where $f$ is a sufficiently
smooth function of its finite dimensional arguments, then the abstract hypotheses
for $\F$ are satisfied.

If the parameters $\nu, \frak{m}$ are such that the operator
 $\LL$ is boundedly invertible (we will indicate the
explicit spaces later),
we rewrite \eqref{operator} as
\begin{equation} \label{contraction}
u(x)=\epsilon\LL^{-1} \F(u)(x)\equiv \mathcal{T}(u)(x).
\end{equation}

\subsection{Remarks on spaces}
We see that, to apply the above program, it is useful to formulate
the problem in spaces of functions
that satisfy the following properties (there are links among these properties as we
will see in the concrete examples in Section~\ref{space}).
\begin{itemize}
\item
Consist of analytic functions (or functions with a specified
regularity).
\item
The
norms can be read off from the Fourier coefficients.
\item
It is possible to give estimates of the composition on the right with $f$
under regularity properties in $f$.
\item  It  is
possible to obtain Lipschitz estimates of the operator composing
 with $f$ (such operators are often called ``Nemitski operators'',
``left composition operators'' or ``nonlinear superposition operators''
\cite{AZ90, IKT13}).
\item
These spaces are Banach algebras under
pointwise multiplication.
\item
These spaces are Hilbert spaces (so that we can take advantage of
selfadjointness of some operators and use sharp results in spectral theory).
\item
The operators we consider   that are diagonal with real eigenvalues are
selfadjoint.
\end{itemize}

Some spaces that satisfy these conditions are introduced in
Section~\ref{space}. These spaces are inspired by the Bargman spaces
used in quantum field theory and in complex analysis. They have
already been used in other papers \cite{CallejaCL13,CCCdlL17,WangL20}. We
note  that, in comparison with  the papers above, we present
Lemma~\ref{Banachalgebra} that shows that the good properties of these
spaces are valid for a larger range of parameter values than those
considered in \cite{CCCdlL17,WangL20}. Hence, the results in the
above papers can be extended slightly.

\medskip
\subsection{The alternative method used in the elliptic case}
If the parameters $\nu, \frak{m}$ are such that the operator $\LL$ has zero
eigenvalue, we use the classical
alternative method of
bifurcation theory. In the case that the eigenvalues are
simple, this method was considered in \cite{CrandallR}. In our case,
the eigenvalues have always higher multiplicity, hence, we will follow \cite{Cesari, ChowH, AmbrosettiA, IossJ, ArnoldAI99}.

\def\LLo{{ \mathcal{L}_{\nu,\mathfrak{m}_{0} } }}
For fixed $\frak{m}_{0}>0,$ we denote by
\[
\LL:=\LLo+(\frak{m}-\frak{m}_{0}),
\]
where
\[
\LLo:=\sum_{i=1}^{d}\nu_{i}^{2}\frac{\partial^{2}}{\partial x_{i}^{2}}+\frak{m}_{0}
\]
has zero eigenvalue, and we call $(\frak{m}-\frak{m}_{0})$ the bifurcation parameter.

We realize
that, since $\LLo$ is self-adjoint (again, we will specify the
appropriate spaces later),
its kernel and the closure of its range are orthogonal.
Since its spectrum is discrete, we can define spectral projections
on the kernel and the range of $\LLo$. We call attention that
we only use the operator for $\mathfrak{m} = \mathfrak{m}_0$.

We will denote by $\Pi_K, \Pi_R$ the projections  on the
kernel and the closure of the range, respectively. These projections are
complementary (i.e. $\Pi_K + \Pi_R = \text{Id}$ ) and orthogonal.

Therefore, the equation
\eqref{operator} is equivalent to
the system of equations obtained taking projections
of \eqref{operator}  on the kernel
and on  the range. Introducing, furthermore, the notation
\def\hu{\hat u}
\def\ou{\overline u}
\def\hv{\hat v}
\def\ov{\overline v}
\[
\hu = \Pi_R u, \,\ou = \Pi_K u
\]
(so that $u = \hu + \ou$),
then \eqref{operator} can be rewritten as $:$
\begin{equation} \label{alternative}
\begin{split}
&  (\frak{m}-\frak{m}_{0})\ou=  \epsilon\Pi_K \F(\hu + \ou), \\
& \Pi_R \LLo \hat{u} = -(\frak{m}-\frak{m}_{0})\hu+\epsilon
\Pi_R\F(\hu + \ou) .
\end{split}
\end{equation}

Furthermore, when $ \LLo^R \equiv \Pi_R \LLo \Pi_R$ is boundedly invertible
as an operator on the closure of the range  of $\LLo$,
we have that \eqref{alternative} is equivalent to
\begin{equation} \label{alternative2}
\begin{split}
&  (\frak{m}-\frak{m}_{0})\ou = \epsilon\Pi_K \F(\hu + \ou), \\
& \hu  =  (\LLo^R)^{-1}(-(\frak{m}-\frak{m}_{0})\hu+\epsilon \Pi_R\F(\hu + \ou))
\end{split}
\end{equation}
The system \eqref{alternative2} is a system for the unknowns $\hu, \ou$.

The first equation in \eqref{alternative2} is often called
the \emph{``bifurcation equation''} and  the second one is  called  the \emph{``range equation''}.

The classical method, which we will follow,  to analyze
\eqref{alternative2} is to, for a given $\ou$, find a $\hu(\ou, \epsilon)$
that solves the range equation. This will be an easy application of
the contraction mapping. Once we have obtained such $\hu(\ou, \epsilon)$,
the bifurcation equation becomes an equation for $\ou$ alone,
namely
\begin{equation}\label{ou}
 (\frak{m}-\frak{m}_{0})\ou=\epsilon \Pi_K \F( \hu(\ou, \epsilon)  + \ou ).
\end{equation}

Since $\ou$ is a finite-dimensional variable, the
equation \eqref{ou} is a finite-dimensional equation, which can be
analyzed using the methods of singularity theory.

The interesting cases are when the nonlinearity is at least quadratic
in the known. The linear terms can be absorbed in the linear
part.

This equation \eqref{ou} will,  under some explicitly non-degeneracy conditions
which depend only on the derivatives w.r.t. $\epsilon$ of the left
hand side of \eqref{ou},
have several branches of solutions and require
somewhat
complicated non-linear analysis, but it is a finite-dimensional problem.
The study of the branches etc. involves some assumptions on the nonlinearity
$f$. To analyze the bifurcation equation, there
are several methods in the literature.

a) Using the jets of the equation to apply a degenerate implicit function theorem.

b) Using some fixed point theorem based on index theory.

These  methods, of course require some non-degeneracy assumptions but give
 very precise information on the detailed nonlinearities, which are
affected by the symmetry etc. We refer to the references above for
the rich mathematical results and applications of singularity theory and bifurcation theory.

In this paper, we will just discuss a very simple explicit nonlinearity
and  show that putting together bifurcation theory and
the fixed point theory,
for all small enough $\epsilon$, we can analyze all
the possible ranges of the parameters $\nu, \frak{m}$ in  \eqref{eq0}.
In this example, the bifurcation theory gives an explanation why
the fixed point method breaks down. Indeed, when the parameters of the problem
are close to the resonant values, there are several small solutions.

\bigskip
\subsection{The ill-posed evolution problem under nonresonance}\label{illnonresonance} For the evolution equation \eqref{evolution},
we are interested in finding quasi-periodic solutions of the form
$u(t,x)=U(\omega t,x)$ with frequency $\omega\in \mathbb{R}^{b}$,
where $U: \mathbb{T}^{b}\times \mathbb{T}^{d}\rightarrow\mathbb{R}$
is the hull function of the solution  $u.$

We will present two different ways of analyzing the equation
\eqref{evolution}. A fixed point analysis and  method
based on reduction to time-dependent center manifolds
\footnote{The fixed point analysis will allow nonlinearities
that loose 2 derivatives, whereas the reduction to center
manifolds allows to loss of $(2 - \kappa)$ derivatives. }.

\def\Q{{ \mathcal{Q}_{\omega,\nu,\mathfrak{m} } }}
We denote by $\Q$ the linear operator
\begin{equation}
\mathcal{Q}_{\omega,\nu,\frak{m}}:=(\omega\cdot\partial_{\theta})^{2}+\sum_{i=1}^{d}\nu_{i}^{2}\frac{\partial^{2}}{\partial x_{i}^{2}}+\frak{m},
\end{equation}
where $\nu=(\nu_{1},\nu_{2},\cdots,\nu_{d})\in[1,2]^{d}.$

\subsubsection{A fixed point analysis}\label{illfixed}
In the fixed point method, we allow that the nonlinearity loose 2 derivatives with the following form:
\[
\mathcal{N}(U)(\theta,x):=f(\theta,x,U(\theta,x),D_xU(\theta,x),D^2_xU(\theta,x)),
\]
then, \eqref{evolution} becomes
\begin{equation}\label{evolution1}
\Q U=\epsilon \mathcal{N}(U).
\end{equation}
We
notice that $\Q$ is a diagonal operator in the Fourier basis, i.e.
\[
\mathcal{Q}_{\omega,\nu,\frak{m}}\{\exp\{\mathrm{i}(l\theta+kx)\}\}
=\Upsilon_{l,k}\{\exp\{\mathrm{i}(l\theta+kx)\}\}
\]
with
\begin{equation}\label{newupsilon}
\Upsilon_{l,k}:=-\langle \omega, l\rangle^{2}-\sum_{i=1}^{d}\nu_{i}^{2}k_{i}^{2}+\frak{m}.
\end{equation}
The problem with the analysis of the multipliers \eqref{newupsilon}
is that when $b$, the dimension of the frequencies, is bigger than
$1$, the set
$\{ \langle \omega, l \rangle \}_{l \in \ZZ^b} $ is dense on the reals.
Hence, $ \{ \langle \omega,   l \rangle^2 \}_{l \in \ZZ^b}$ is dense on
$\real_+$.

\medskip
The hypotheses of the fixed point approach are as follows:

Suppose that the parameters $\omega,\nu,\frak{m}$ meet one of the following hypotheses:

(H1) The value of $-\sum_{i=1}^{d}\nu_{i}^{2}k_{i}^{2}+\frak{m}$ is negative;

(H2) The value of $-\sum_{i=1}^{d}\nu_{i}^{2}k_{i}^{2}+\frak{m}$ is positive and $\omega$ is 1-dimensional $\&$ $\omega\in [1,2]\subset \mathbb{R}^{1}.$

In these two
cases, we can still use the freezing of coefficient method used for the elliptic case away from resonances to obtain the
solutions. Otherwise,  the  freezing of coefficient  method fails to solve
\eqref{evolution} and we have to resort to the method
described in the next Section.

\subsubsection{Time-dependent center manifolds}
A method of wider applicability (and which produces
solutions more general than response solutions)
is to apply  a  time-dependent  center manifold theorem.

We will allow that the nonlinearity $f$ depends on $D_{x}U$. More
generally that the nonlinearity is given by
a functional which looses $(2-\kappa)$ derivatives.

The recent paper \cite{CdlLR20} develops a time-dependent center manifold
theory that applies to ill-posed equations. More precisely the methods of \cite{CdlLR20}
requires that  $\mathcal{N}(U)$ is  several times differentiable
from a space of functions having $r$ derivatives to a space of functions
having $(r-2+\kappa)$ (for some $\kappa>0$) derivatives
\footnote{We do not know whether the requirement  of $\kappa > 0 $
is really needed of it is a limitation of the method.}.

  Note that, when the freezing of coefficient method applies,
we do not need to include the $\kappa$, and we could
obtain results for  nonlinearities that loose $2$ derivatives.
Furthermore, the results on the elliptic case, require only that the
nonlinearity is Lipschitz, but for the center manifold, we will need
that the nonlinearity is several times differentiable.

Since the results for nonlinearities that loose $(2 - \kappa)$ derivatives
can be obtained just directly from \cite{CdlLR20}, in this paper
we will present only the results for nonlinearities that loose $1$ derivative
and present full details in  this case. As we will see, dealing with
the case that the nonlinearity looses one derivative, is simpler
than the case discussed in \cite{CdlLR20}. We hope that the present
simple proof can be pedagogically motivating for these areas of results.
Of course, in the classical problems in which the losses of
derivatives are caused by applying differentials, the loss of derivatives
are integers and, loosing one derivative is the best that one can do
in this classical case.

We will show that the results of \cite{CdlLR20} apply to
\eqref{evolution}. Then, we conclude  that even if \eqref{evolution}
is ill-posed, there is a \emph{finite-dimensional}
manifold evolving quasiperiodically which is invariant under
\eqref{evolution}.

Once this center manifold
is established, one can use finite-dimensional methods to obtain a
varieties of solutions: response subharmonics, (un)stable manifolds,
etc. by a finite computation.

More precisely:  By using a quasi-periodic
parameterization of the center manifold, we are reduced to studying
a finite-dimensional non-autonomous differential equation.
We will provide the first terms in the expansion these
manifold and  recall that there are many results in the literature
of finite-dimensional system which allow to conclude that, if
the perturbations of the system satisfy some concrete non-degeneracy
assumptions, then the perturbed system admits interesting
orbits.

For example, using KAM theory, one can get response solutions
or solutions with external and inner frequencies.  If these periodic
solutions have positive Lyapunov exponents (which can be computed
perturbatively), one can produce stable manifolds, using Melnikov theory,
one can get subharmonic quasi-periodic orbits (and possibly their
stable/unstable manifolds). There are many such results in the finite-dimensional theory that give precise conditions for the persistence of
orbits of some kind if the perturbations satisfy different conditions.

We will not present these finite-dimensional results in detail since they are well established \cite{Mitropolsky,Melnikov}
and their methodology is rather different from the main thrust of this
paper and its main use is applications to concrete models.

\begin{remark} \label{encouragement}
The paper \cite{CdlLR19} considers ill-posed time-independent
manifolds and shows that there are infinite-dimensional
manifolds of solutions that converge to them.

It seems likely that one can adapt the proofs of
existence of stable manifolds  for autonomous
models to the non-autonomous cases considered here. If
such adaptation was possible,
 besides the finite-dimensional families  solutions produced here,
one would get infinite-dimensional families asymptotic
to them in the future (or in the past).

Of course, adapting the results of \cite{CdlLR19} to
the time-dependent case in \cite{CdlLR20}
would have several other applications.
\end{remark}

\section{Function spaces}\label{space}

\begin{definition}
Given $\rho>0,$ we introduce the complex torus $\mathbb{T}_{\rho}^{d}:$
\[
\mathbb{T}_{\rho}^{d}:=
\{
x\in\mathbb{C}^{d}/(2\pi\mathbb{Z})^{d}: \mathrm{Re}(x_{j})\in \mathbb{T},\,\, |\mathrm{Im} x_{j}|\leq \rho,\,\,j=1,\cdots,d
\}.
\]
\end{definition}

Note that $\torus^d_\rho$ can be considered as a $2d$ real manifold with boundary.

For a  function $u: \mathbb{T}_{\rho}^{d}\rightarrow\mathbb{C},$ we denote its Fourier expansion:
\[
u(x)=\sum_{k\in\mathbb{Z}^{d}}\hat{u}_{k}e^{\mathrm{i}k\cdot x},
\]
where $k\cdot x=\sum_{i=1}^{d}k_{i}x_{i}$ and  $\hat{u}_{k}$ are the Fourier coefficients of $u.$

If $u$ is analytic and bounded on $\mathbb{T}^{d}_{\rho}$, then the Fourier coefficients satisfy the Cauchy bounds:
\[
  |\hat{u}_{k}|\leq
  e^{-|k|\rho} \max_{x\in\mathbb{T}_{\rho}^{d}}|u(x)|
\]
with $|k|=\sum_{i=1}^{d}|k_{i}|$.

\bigskip
The spaces we will work with are:

\begin{definition}
For $\rho\ge 0, r\in \mathbb{Z}_{+}$, we denote by $H^{\rho,r}$

\begin{equation*}
\begin{split}
H^{\rho,r}:&=H^{\rho,r}(\mathbb{T}^{d}_{\rho})\\
&=\left\{u: \mathbb{T}_{\rho}^{d}\rightarrow\mathbb{C} \,\,\big|\,\, \|u\|_{\rho,r}^{2}=\sum_{k\in\mathbb{Z}^{d}}|\hat{u}_{k}|^{2}e^{2|k|\rho}(1+|k|^{2})^{r}<+\infty\right\}.
\end{split}
\end{equation*}

\end{definition}
Note that $(H^{\rho,r},\|\cdot\|_{\rho,r})$ is a Hilbert space.

\begin{remark}
When $\rho=0,$
$H^{r}(\mathbb{T}^{d}):=H^{0,r}(\mathbb{T}^{d})$ is the standard
Sobolev space. According to the Sobolev embedding theorem, we have
that the space $H^{r+\lambda}(\mathbb{T}^{d})$ $(\lambda=1,2,\cdots)$
is continuously embedded into $C^{\lambda}(\mathbb{T}^{d})$ for
$r>d/2$ (see \cite{Tay11b}).

When $\rho>0,$ the space $H^{\rho,r}$ is a closed space of standard
Sobolev space $H^{r}(\mathbb{T}_{\rho}^{d}),$ which consists of
complex analytic functions.
\end{remark}

The spaces $H^{\rho, r}$ enjoy many remarkable properties.
We have collected the  ones we will use in Appendix~\ref{sec:properties}.
The most important ones are the properties of the operator given
by composition in the left. See Lemma~\ref{compos1}.  This will
justify that the  operator $\F(u)= f(x,u, Du, D^2u)$  satisfies
the abstract properties when the function $f$ is analytic (or
sufficiently differentiable).

\smallskip
The following result is a straightforward consequence
of the fact that the norms in the $H^{\rho,r}$ spaces are
weighted sums of the Fourier coefficients.

\begin{proposition}
\label{operatornorm}
(\cite{CCCdlL17})If we have a linear operator $\mathcal{A}$ which is diagonal in the Fourier basis,
\[
\mathcal{A}\exp\{\mathrm{i}kx\}
=\Upsilon_{k}\exp\{\mathrm{i}kx\},
\]
for suitable coefficients $\Upsilon_{k},$

\[
\|\mathcal{A}\|_{H^{\rho,r}\rightarrow H^{\rho,r}}\leq  \sup_{k}|\Upsilon_{k}|.
\]

More generally, if
\[
  |\Upsilon_{k}| \le C(1 +|k|^2)^{-\lambda/2},
\]
then,
\[
\|\mathcal{A}\|_{H^{\rho,r}\rightarrow H^{\rho,r+\lambda}}\leq
\sup_{k}|\Upsilon_{k}|(1 + |k|^2)^{\lambda/2}.
\]
\end{proposition}

\section{Nonresonant case}\label{nonresonant}

In this section, we give a simple proof of the main result in \cite{Shi}
but we weaken the assumption that the nonlinearity $f(x,u)$ is
trigonometric polynomial. We allow that the nonlinearity is
$\F(u)(x)=f(x,u,Du, D^2u)$ with $f$ analytic or finitely differentiable.

We now present our result for the model \eqref{operator} in elliptic case far
from resonances.

\begin{theorem}\label{analyticThm}

For fixed $\frak{m}>0,$ any $0<\delta\ll1$, given $\rho \ge 0, r-2>d/2,$ let
$B_{s}(0)\subset H^{\rho,r}$ be closed ball around the origin with the radius $s>0$.

Assume that $\F$  is Lipschitz from  $B_s(0)$ into $H^{\rho, r-2}$.

There exist $\epsilon_{*}>0$ depending on $\nu, \frak{m}, \delta, s,
\Lip(\F)$ and a set
$I\subset [1,2]^{d}$ with Lebesgue measure
$\mathrm{mes}(I)=O(\delta),$ such that when $0<\epsilon<\epsilon_{*}$, for any $\nu\in
[1,2]^{d}\backslash I,$ the equation \eqref{operator} admits a unique solution $u(x)\in B_{s}(0).$
\end{theorem}

\begin{remark}\label{solutionsanalytic}
When $\rho >0$, the solutions produced by
Theorem~\ref{analyticThm} will be analytic as functions of
their arguments. When $\rho = 0$, the
solutions will be in the classical Sobolev space and finitely differentiable.
\end{remark}

\begin{remark}\label{parameters}
Since we are just applying the contraction mapping
principle if $\F$ is differentiable (or analytic with respect to
parameters), the solutions produced by Theorem~\ref{analyticThm} will
depend differentiably (or analytic) in parameters.  The analyticity with respect
to parameters is very natural when we consider the nonlinearities of
the form~\eqref{concreteF}.
\end{remark}

We will show in Lemma~\ref{compos1}   that if
$f$ is analytic in a small ball of its arguments, $\F$ is indeed Lipschitz (analytic)
on $B_s(0)$ of the spaces $H^{\rho, r}$ for any $\rho \ge 0$, $r -2 > d/2$.
If $f$ depends analytically on the parameters,
then the function $\F$ is also analytic in the sense of
analytic functions from one Banach space to another. See \cite[Chapter III]{HilleP57}
for more details on the theory of analytic functions from a Banach space
to another.

We also show in Lemma~\ref{compos1}
that if $f$ is $C^{r+1}$, the function $\F$  is Lipschitz from
$H^{r}$ to $H^{r-2}$ for $r-2 > d/2$.

To prove Theorem
\ref{analyticThm}, we first prove that the operator
$\LL$ is boundedly invertible from $H^{\rho, r-2} $ to $H^{\rho,r}$. These are,
of course, standard elliptic estimates that show that inverting the
operator \emph{gains two derivatives} \cite{Agmon,Tay11b}.

We note that if $\rho>0$ or $\rho=0,$ $r>d/2+2,$ the solutions produced here satisfy the equation \eqref{eq0} in the classical sense.

\subsection{Estimates on the inverse operator $\LL^{-1}$}
\label{I}

First, we give the measure of the parameter set in $\nu$ space which produces resonance term.

\begin{lemma}\label{mesI}
For sufficiently small $\delta>0,$ $d\in\mathbb{Z}_{+},$ fixed $\mathfrak{m}>0,$
we define the following parameter set of $\nu$:
\[
I=\left\{
\nu\in[1,2]^{d}\,\,|\,\,\exists\,\, k\in \mathbb{Z}^{d},\,\,such\,\, that\,\,
 \left|-\sum_{i=1}^{d}k_{i}^{2}\nu_{i}^{2}+\mathfrak{m}\right|\leq\delta
\right\}.
\]
Then, we have $\mathrm{mes}(I)=O(\delta),$
where $O(\delta)$ is the same order of $\delta.$
\end{lemma}

\begin{proof}
When $k=0,$ $\mathrm{mes}(I)=0.$
When $k\in\mathbb{Z}^{d}\backslash \{0\},$ for fixed $\mathfrak{m}>0,$
we define the set
\[
K=\left\{k^{(n)}\in\mathbb{Z}^{d}\setminus\{0\}\,\,|\,\,
\sum_{i=1}^{d}(k^{(n)}_{i})^{2}\nu_i^2=\mathfrak{m}, \,\, n=1,\cdots,N
\right\}.
\]
We choose $k^{(n)}\in K$ and define the set of $\nu$ as
\[
I_{n}=\left\{\nu\in [1,2]^{d}\,\,\big|\,\,|-F_{k^{(n)}}(\nu)+\mathfrak{m}|\leq\delta\right\}
\]
with $F_{k^{(n)}}(\nu)=\sum_{i=1}^{d}\big(k_{i}^{(n)}\big)^{2}\nu_{i}^{2}.$
Then,
\[
\mathrm{mes}(I_{n})\leq\frac{2\delta}{\inf\{|\nabla F_{k^{(n)}}(\nu)|_{2}\}}
\leq\delta,
\]
where
$
\nabla F_{k^{(n)}}(\nu)=\left(2(k_{1}^{(n)})^{2}\nu_{1}, 2(k_{2}^{(n)})^{2}\nu_{2},\cdots, 2(k_{d}^{(n)})^{2}\nu_{d}\right),
$
$|\cdot|_{2}$ is $l^{2}$-norm.

Thus, we have
\[
\mathrm{mes}(I)=\mathrm{mes}\bigg(\bigcup_{n=1}^{N}I_{n}\bigg)\leq\sum_{n=1}^{N}\mathrm{mes}(I_{n})= O(\delta).
\]
\end{proof}

According to the Lemma, we conclude that when
$\nu\in [1,2]^{d}\setminus I,$ the diagonal operator $\mathcal{L}_{\nu,\frak{m}}$  does not have zero
eigenvalues and that the absolute value of the eigenvalues is bounded from
below by $\delta$.

We furthermore observe that for large $k$, the eigenvalues are
bounded from below by $C(1 + |k|^2)$. If the eigenvalues
do not vanish, we have  a bound
$|\Upsilon_{k, \mathfrak{m} }|  \ge C(1 +|k|^2)$.
Therefore, by Proposition~\ref{operatornorm}, one has the following proposition:
  \begin{proposition}\label{estimatesLLm1}
   If
    $\nu\in [1,2]^{d}\setminus I,$ we have
    \begin{equation*}
      \begin{split} \label{operatorestimates}
      & \| \LL^{-1} \|_{H^{\rho, r} \rightarrow H^{\rho,r} } \le \delta^{-1},\\
      & \| \LL^{-1} \|_{H^{\rho, r-2} \rightarrow H^{\rho,r} } < C
      \end{split}
     \end{equation*}
     for some constant  $C$.
\end{proposition}

\subsection{Existence of solutions}

Recall the equation \eqref{contraction}
\begin{equation*}
u(x)=\epsilon\LL^{-1}\F(u)(x)\equiv \mathcal{T}(u)(x).
\end{equation*}
We will use the contraction mapping principle.

We assume that there is a closed  ball $B_{s}(0)$ around
the origin in $H^{\rho,r}$ with radius $s>0$, where
\[
B_{s}(0)=\{u\in H^{\rho,r}\,\,\big|\,\,\|u\|_{\rho,r}\leq s\},
\]
such that $\F$ is defined in $B_s(0) $ as an operator
from $B_s(0) \subset H^{\rho, r}$ to $H^{\rho, r-2}$. Moreover, we  assume that $\F$ is Lipschitz in $B_s(0)$  as an operator from
the $H^{\rho, r}$ to $H^{\rho, r-2}$. That is
\[
  \| \F(u_1) - \F(u_2) \|_{\rho, r-2} \le \Lip(\F) \| u_1 - u_2 \|_{\rho, r}.
  \]

  We will show that, for sufficiently small $\ep$, the operator
  $\mathcal{T}$ maps the ball $B_s(0)$ into itself and
  is a contraction.  Therefore, it has a unique fixed point in this ball.

Denote $\epsilon_{*}=\min\{1/(2 C\Lip(\mathcal{F})), s/ (2C\|\mathcal{F}(0)\|_{\rho,r-2})\}$.
When $0<\epsilon<\epsilon_{*},$
for any $u_{1}, u_{2}\in B_{s}(0),$ one has
\begin{equation*}
\begin{split}
\|\mathcal{T}(u_{1})-\mathcal{T}(u_{2})\|_{\rho,r}
&=\|\epsilon \LL^{-1}( \F(u_1)- \F(u_2) ) \|_{\rho, r} \\
&\leq\ep \|\LL^{-1} \|_{H^{\rho,r-2}\rightarrow H^{\rho,r}} \| \F(u_1) - \F(u_2) \|_{\rho, r-2} \\
& \le \ep \| \LL^{-1}\|_{H^{\rho,r-2}\rightarrow H^{\rho,r}} \Lip(\F) \|u_1 - u_2 \|_{\rho, r}\\
&\leq \frac{1}{2}\|u_1 - u_2 \|_{\rho, r}.
\end{split}
\end{equation*}
For $u\in B_{s}(0),$ one has
\begin{equation*}
\begin{split}
\|\mathcal{T}(u)\|_{\rho,r}
&=\|\mathcal{T}(0)+\mathcal{T}(u)-\mathcal{T}(0)\|_{\rho,r}\\
&\le \ep  \|\LL^{-1} \|_{H^{\rho,r-2}\rightarrow H^{\rho,r}}\|\F(0) \|_{\rho,r-2} +\frac{1}{2}s \\
&\leq s.
\end{split}
\end{equation*}
It follows from  the contraction principle that there exists a unique solution $u(x)\in B_{s}(0)$ belonging to $H^{\rho,r}$ for \eqref{operator}.

One could think of optimizing the choice of $\ep_*$ to obtain
uniquness in a larger ball. A different optimization is to locate
the solution in a smaller ball.

\section{Resonant case}\label{resonant}
In this section,  we
study the case when the operator $\LLo$ has a nontrivial
kernel (this case was not considered in \cite{Shi}).

We will explain the general theory and give a concrete example when
the nonlinearity is just $f(u) = u^2$.
Note that, in this case, the forcing is identically zero at the origin, so that the
solutions produced by Theorem~\ref{analyticThm} are just $0$. We will show
that, near the bifurcation points, there are some solutions
besides those produced by Theorem~\ref{analyticThm}.

The nonlinearity $u^2$ has been chosen for simplicity. As we will see
in Remark~\ref{higherorder},  the same results apply for all linearities
which vanish to second order in $u$. Notice that adding nonlinearities
with a nonvanishing linear term in $u$ is better dealt with by changing the
linear operator we are considering.

 We first rescale the original system to get a slow system. Let $v=\epsilon u$.
Then the equation \eqref{evolution} becomes
\begin{equation}\label{problem}
\LL v(x)= v^{2}(x).
\end{equation}
Equation~\eqref{problem}  can  be rewritten as
\begin{equation}\label{newproblem}
(\LLo +(\mathfrak{m}-\mathfrak{m}_{0}))v(x)= v^{2}(x).
\end{equation}

We will show that for values $\frak{m}$ close to $\frak{m}_0$,
the problem   \eqref{problem} may have several small solutions.
These solutions are functions of the
 \emph{bifurcation parameter} $\mathfrak{m} - \mathfrak{m}_{0}$.
Clearly, the existence of several small solutions establishes
that one cannot apply the contraction mapping principle.

\begin{remark}\label{symmetry}
	Notice that if $v:\mathbb{T}^{d}\rightarrow\mathbb{R}$ satisfies \eqref{problem}, then for any $x_*\in\mathbb{R}^{d},$ the functions
	\[v_{x_{*}
}(x)=v(x+x_{*})\]
are also the solutions so that  we always
 obtain $d$-dimensional families of solutions.

Notice also that the dimension of the kernel is  expected to
be  $2^d$ (see Remark~\ref{flipping})
and $2^d > d$.  So that, when the dimension grows,
the dimension of the kernel grows much faster than the dimension of
the families.
\end{remark}

We anticipate that the main
difficulty is that the kernel will be high dimension and that,
at the same time, there are symmetries. Since the dimension of the
kernel grows exponentially with the dimension $d$
and the dimension of the symmetry is $d$, we will restrict
our study to $d = 1,2$ and only make some remarks about $d \ge 3$.
\medskip

We give the following result:

\begin{theorem}\label{thm:bifurcation}
Consider the problem \eqref{newproblem}.

Assume that the dimension $d$ of the space is either $1$ or $2$ and that the parameter $\nu$ satisfies Assumption~\ref{simplicity}.
Assume also that the parameter $\nu$ does
not belong to another set of measure zero  on
which an explicit rational function vanishes.

Let $\frak{m}_0$ be such that  the   operator $\LLo$ has
a nontrivial kernel, which by Assumption~\ref{simplicity},
has dimension $2^d$ and consists of exponentials of
wave vectors which are obtained by changing signs in a vector.

Let $\sigma$  be the sign of an explict formula
$A+B$ (given in \eqref{AB} in  two dimensions).

Then for $\frak{m}$ sufficiently close to (not equal to) $\frak{m}_0$ ,
and with $\frak{m} - \frak{m}_0$ having the  sign $\sigma$,  the problem \eqref{newproblem} admits $d$-dimensional families of non-zero solutions.

Moreover, these solutions are functions of the
bifurcation parameter $(\mathfrak{m} - \mathfrak{m}_{0})$ and are
analytic in  $|\mathfrak{m} - \mathfrak{m}_{0}|^{\frac{1}{2}}$.
\end{theorem}

\begin{remark}\label{m-m0}
Notice that the branches of \eqref{newproblem} exist for both  intervals $|\frak{m}-\frak{m}_{0}|$ small and satisfying that  when $d=1$, $\frak{m}-\frak{m}_{0}$ is positive, and when $d=2,$ $\frak{m}-\frak{m}_{0}$ having the same sign with $A+B$,  which is denoted by $\sigma$ in Theorem~\ref{thm:bifurcation}.

The cases when the extra branch appears for $\frak{m} - \frak{m}_0 < 0$ are
called \emph{subcritical bifurcation} and the cases when
the extra branch of solutions appears for  $\frak{m} - \frak{m}_0 > 0$
are called \emph{supercritical bifurcation}.

Most of the classical bifurcation theory is concerned with existence
of stationary solutions, but in this case, due to the symmetries
of the problem, we obtain always families of equilibria.
\end{remark}

The proof of Theorem~\ref{thm:bifurcation}
is based on the alternative method of bifurcation theory explained
before. In this section, we give all the needed details for
the specific nonlinearity $f(u)=u^2$.

From the analysis in Section \ref{mainidea}, we know that equation \eqref{newproblem} can be regarded as the \emph{``bifurcation equation''} and the \emph{``range equation''}, respectively:
\begin{equation}\label{B}
  (\frak{m}-\frak{m}_{0})\ov=\Pi_K (\hv + \ov )^{2},\quad\quad\quad\,\,\,
\end{equation}
\begin{equation}\label{R}
\big(\LLo+(\frak{m}-\frak{m}_{0})\big)\hv = \Pi_R (\hv + \ov)^{2}.
\end{equation}
\subsection{Some general results on the range equation}
The range equation is much simpler to deal with than
the bifurcation equation and admits a
general theory.  As we will see, it admits solutions which
are analytic in $\ov$ under rather general circumstances.

For notational convenience later, we will introduce some coordinates
\[\alpha=(\cdots, \alpha_{j},\,\cdots),\,\,j=1,\cdots,2^d\]
in the kernel.  So that $\ov= \sum_j \alpha_j\exp(\mathrm{i} k^j \cdot x)$
for the $k^j$ the wave numbers in the kernel of $\LLo$.
Then, we will show that the $\hv$ solving the range equation is  analytic  as the function of $\alpha$.

Since we are going to do algebra of polynomials,
it is convenient to think of $\alpha$ as complex numbers
(even if the problem at hand is real). It also allows us to
present the eigenfunctions as exponentials rather than as
pair of $\sin/\cos$. This is useful when considering products.

The counting of dimensions is slightly delicate:
If  we want to have real solutions, we will need
that if $k = -k$, the corresponding coefficients
satisfy $\alpha = \alpha^*$.  The complex dimension of the space of
$\alpha$ is the dimension of the kernel. If we require  that the
solutions are real, the complex dimension will be half  the dimension
of the kernel, which corresponds to  having a real dimension equal
to the dimension of the kernel.

Consider the range equation \eqref{R} in the space $H^{\vartheta,\rho,r}$ with $0<\vartheta\ll 1$, where
\begin{equation*}
\begin{split}
H^{\vartheta,\rho,r}=\big\{v:\mathbb{C}^{2^d}\rightarrow H^{\rho,r}\,\,\big| \,\,v(\alpha)=\sum_{j=0}^{\infty} v_j \alpha^j,\,\|v\|_{\vartheta,\rho,r}=\sum_{j=0}^{\infty} \|v_j\|_{\rho,r}\vartheta^j<\infty\big\}
\end{split}
\end{equation*}
is a Banach algebra. Then we have the following results for any dimension $d$.
\begin{lemma}\label{convergent}
	There exists a solution $\hat{v}\in H^{\vartheta,\rho,r}$, for the range equation \eqref{R}, which is analytic in $\alpha\in \mathbb{C}^{2^d}$.
	\end{lemma}
\begin{proof}
	We assume $\mathfrak{m}-\mathfrak{m}_{0}=O(\epsilon).$	Then, range equation \eqref{R} can be rewritten as:
	\begin{equation}\label{R1}
\hv = \LLo^{-1}\big(-\epsilon\hv+\Pi_R (\hv + \ov)^{2}\big)\equiv T(\hv).
	\end{equation}
It is easy to see that  the operator $T$ defined in \eqref{R1} maps the space $H^{\vartheta,\rho,r}$ into itself.

We choose a ball $\mathcal{B}_{s}(0)\subset H^{\vartheta,\rho,r}$. The remaining task is to verify that the operator  $T$ maps the ball into itself and it is a contraction in this ball.

Since $\ov=\sum_{j=1}^{2^{d}}\alpha_{j}\exp{\mathrm{i}k^j x}$, belonging to the kernal space of $\LLo$,  only contains finite terms,  there exists a constant $C$ such that
\[\|\ov\|_{\vartheta,\rho,r}\leq C\vartheta.
\]
Therefore, one has
\[\text{Lip}(T)\leq C(\epsilon +\vartheta+s)\leq \frac{1}{10}\]
when we choose $\epsilon$ small enough and $s<\frac{1}{20C}-\vartheta$. This reveals that $T$ is a contraction in the ball $
	\mathcal{B}_{s}(0)$.
On the other hand,
	for  $U\in \mathcal{B}_{s}(0)$ with $s$ chosen above, one has
	\begin{equation*}
	\begin{split}
	\|T(\hv)\|_{\vartheta,\rho,r}
	\leq\|T(0)\|_{\vartheta,\rho,r}+ \|T(\hv)-
	T(0)\|_{\vartheta,\rho,r}
	\leq C\vartheta+
	\frac{1}{10}s\leq s
	\end{split}
	\end{equation*}
by  choosing the radius $s$ satisfies $\frac{10C}{9}\vartheta\leq s<\frac{1}{20C}-\vartheta$.
	
		In conclusion, by the fixed point theorem in the Banach space  $H^{\vartheta,\rho,r}$,
	there exists a unique solution $\hv\in H^{\vartheta,\rho,r}$ analytic in $\alpha$ for the  equation \eqref{R1}.

	\end{proof}
\begin{remark} \label{higherorder}
Note that the proof of Lemma~\ref{convergent} works even if the
nonlinearity is an analytic function  starting with
quadratic terms.

There are also versions of the argument assuming only finite differentiability
of the nonlinearity (acting on spaces of finite-differentiable functions).
Note that, for subsequent use, we only need a finite number of
derivatives. The aim of this paper is not to give a complete coverage,
but to illustrate the possibilities in one example.
\end{remark}
\subsection{Some pereliminary analysis of the bifurcation equation} The fact that $\hv$ is analytic in $\alpha$, shows
that we can write the bifurcation equation \eqref{B} as

\begin{equation} \label{bifurcation}
\ep \alpha =  B(\alpha)\equiv \sum_{j \in \mathbb{Z}^{\text{L}}}
B_j \alpha^j
\end{equation}
where we are using multi-index notation for $\alpha^j$, denoting
by $\text{L}:=2^d$ the dimension of the Kernel of $\LLo$ and
the $B_j$ are (complex) vectors of lenghth $\text{L}$.

In this section, we will present some general results about the
bifurcation equation, which hold for all dimensions of the kernel.

Later on, we will present some complete results for the low dimensional
cases and some remarks that show that the  higher dimensional cases
are more complicated.

We will assume for all subsequent work :
\begin{assumption}\label{simplicity}
The wave numbers of  the eigenfunctions in the
kernel  of $\LLo$ are obtained by changing signs of components of
a vector.
\end{assumption}
As indicated before, for a set of full measure of $\nu$,
Assumption~\ref{simplicity} holds for all  the eigenvalues
of $\LLo$.

An important observation is that, since all the eigenvalues
are exponentials  and the product of eigenvalues is also
an exponential.   The projections over the kernel and the range
are very easy acting on  exponentials. They  either return the same
exponential of zero.
Similarly, we recall that $\LLo$ and $\LLo^{-1}$ acting on exponentials
are just multiplying by a number.

\begin{lemma}\label{noeven}
With the notations of \eqref{bifurcation} and Assumption~\ref{simplicity}.
If $|j|$ is even, then $B_j = 0.$
\end{lemma}
\begin{proof}
We observe that the powers of
$\ov$ have products of $\alpha_j  \exp(\mathrm{i}  k^j \cdot x)$.
So, they are monomials of the form:
\begin{equation}\label{goodform}
B_j \alpha_1^{j_1}\cdots \alpha_{\text{L}}^{j_{\text{L}}}
\exp( \mathrm{i} ( k^1 j_1+ \cdots + k^{\text{L}}j_{\text{L}})\cdot x)
\end{equation}
with the  $k^j$ being wave numbers of functions in the kernel.

The following result is obvious.
\begin{proposition} \label{goodfunction}
We consider the class of functions $\mathcal{G}$ which are
analytic functions of $\alpha$ and all the terms are of
the form \eqref{goodform}.

If $v_1, v_2 \in \mathcal{G}$,
the following belong to $\mathcal{G}$
\[
v_1 + v_2, v_1 \cdot v_2, \Pi_K v_1, \Pi_R v_1, \LLo^{-1} \Pi_R v_1.
\]
\end{proposition}

Hence, it follows  that all  the  terms in
$\Pi_K(\ov+ \hv)^2 $
are of the form \eqref{goodform}.
The bifurcation equation \eqref{bifurcation} is
obtained by equating the coefficents of the
same  exponential functions.

Therefore, the only terms that can appear in the bifurcation equation
are terms in which
$j_1 k^1+ \cdots+ j_{\text{L}} k^{\text{L}} $ is one of the wave numbers in the kernel.

If we look at the first component, the components of
 the $k^1$ are $\pm a$.  A necessary condition for the sum
to be in the kernel is that the first component is $\pm a$.

We just observe that it is impossible to add an even number of
$\pm a$ in such a way that the sum is either $a$ or $-a$.
\end{proof}

\subsubsection{2-dimensional case}
We present our main idea for $2$-dimensional case  under Assumption~\ref{simplicity}.
We remark that the same analysis applies in higher
dimensions when the wave numbers of eigenvalues in the kernel has
only 2 nonzero components.

We denote by
\[\LLo\exp\{\mathrm{i}kx\}=\Upsilon_k\exp\{\mathrm{i}kx\},\] where $\Upsilon_{k}=-\sum_{j=1}^{2}k^2_{j}\nu^2_{j}+\frak{m}_{0}, k\in\mathbb{Z}^{2},$ are the eigenvalues of the operator $\LLo$ and $\exp\{\mathrm{i}kx\} $ are the eigenfunctions corresponding to $\Upsilon_{k}.$
We observe that  the null space of the operator $\LLo$ has complex
dimension 4, and the eigenvectors are $\exp\{\mathrm{i}kx\},$
$k\in \mathcal{K},$ where the set $\mathcal{K}$ is defined as :
\begin{equation}\label{wave}
\begin{split}
\mathcal{K}:=\{k\in\mathbb{Z}^{2}\,\big|\,& k=\{(\pm a,\pm b)\},\,\,a,b\in \mathbb{N}/\{0\},\\ &satisfying\,\,\Upsilon_{k}=0, i.e.,-(a^{2}\nu_{1}^2+b^2\nu_2^2)+\mathfrak{m}_{0}=0\}.
\end{split}
\end{equation}

We denote the null space of operator $\LLo$ by
\begin{equation*}
Ker:=\{\ov_{\alpha}:\mathbb{T}^2_{\rho}\rightarrow\mathbb{C}\,\big|\, \ov_{\alpha}(x)=\sum_{j=1}^{4}\alpha_{j}\exp\{\mathrm{i}k^jx\},
\alpha_{j}\in\mathbb{C},
k^{j}\in\mathcal{K}\}.
\end{equation*}

Note that for real function, without loss of generality, we suppose that  $k^1=-k^4,\,k^2=-k^3$ then $\alpha_4=\alpha_1^*,\,\alpha_3=\alpha_2^*$. Thus we only need to determine $\alpha_1$ and $\alpha_2.$ More precisely, we take $k^1=( a, b),\
 k^2=(a,-b),\,k^3=(-a,b),\,k^4=(-a,-b)$.
\medskip

Now, we go back to the bifurcation equation
\begin{equation}\label{B-2}
(\frak{m}-\frak{m}_{0})\ov_{\alpha}(x)=\Pi_K (\hv_{\alpha}(x) + \ov_{\alpha}(x) )^{2},\,\,
\end{equation}
which can be rewritten as
\begin{equation}\label{B2-dim}
\begin{split}
\epsilon\alpha_l \exp(\mathrm{i} k^l x)=\sum_{|j|\geq 3} B_{j}^l \alpha_{1}^{j_1}\alpha_{2}^{j_2}\alpha_{3}^{j_3}\alpha_{4}^{j_4}\exp( \mathrm{i} ( k^1 j_1+ \cdots + k^{4}j_{4})\cdot x).
\end{split}
\end{equation}
with $j=(j_1,\cdots,j_4)\in \mathbb{N}^4,\,|j|=|j_1|+\cdots|j_4|$. It sufficies to solve the equation for $\alpha$:
\begin{equation}\label{B2-dim1}
\begin{split}
\epsilon\alpha_l =\sum_{|j|\geq 3\atop k^1 j_1+ \cdots + k^{4}j_{4}=k^l} B_{j}^l \alpha_{1}^{j_1}\alpha_{2}^{j_2}\alpha_{3}^{j_3}\alpha_{4}^{j_4}:=B(\alpha),\,\,l=1,\cdots, 4,
\end{split}
\end{equation}
where $B(\alpha)$ satisfies the following proposition.
\begin{proposition}\label{Balpha}
	For $2$-dimensional case, any term in $B(\alpha)$ defined in \eqref{B2-dim1}:

	\begin{enumerate}
		\item  contains a factor $\alpha_l$.
		
		\item the other terms are powers of $|\alpha_1|^2,\,|\alpha_2|^2$.
		
	\end{enumerate}
	
\end{proposition}
\begin{proof}
	It follows from \eqref{wave} that the wave vectors $k^l=(\pm a,\,\pm b)$  with $\Upsilon_{k^l}=0$. 	For simplicity, we denote $k^l$ using the sign only, i.e. $k^1=(+,\,+),\,k^2=(+,\,-),\,k^3=(-,\,+),\,k^4=(-,\,-)$.
	It suffices to prove the case of $k^1=(+,\,+)$
(by defining $a,b$ to have the appropriate sign).

 From $ k^1 j_1+ \cdots + k^{4}j_{4}=(+,\,+)$, one has
	\begin{equation*}
	\left\{
	\begin{array}{l}
	(j_1-j_4)+(j_2-j_3)=1,\\ \\
	(j_1-j_4)+(j_3-j_2)=1.
	\end{array}
	\right.
	\end{equation*}
	This indicates $j_1=j_4+1,\,j_2=j_3$. Therefore, when $l=1$, \[B(\alpha)=\sum_{j_3+j_4\geq 1} B_{j}^1 \alpha_{1}^{j_4+1}\alpha_{2}^{j_3}\alpha_{3}^{j_3}\alpha_{4}^{j_4}
	=\alpha_{1}	\sum_{j_3+j_4\geq 1} B_{j}^1 (|\alpha_{1}|^2)^{j_4}(|\alpha_{2}|^2)^{j_3}.\]
	This concludes our results.
\end{proof}

Proposition~\ref{Balpha} gives that  the bifurcation equation \eqref{B2-dim1} can be represented as:
\begin{equation}\label{fulleq}
\epsilon I=Mz+P(z),\,\,z=(|\alpha_{1}|^2,|\alpha_2|^2)^{T},\,I=(1,1)^{T},
\end{equation}
where $M$ is a $2\times 2$ matrix (will be given later) and $P$ is a homogeneous polynomial of degree 2 or higher
and we have introduced the typographical simplification
$\epsilon = \frak{m} - \frak{m}_0$.

 Note that, introducing $\widetilde{z}=\epsilon z$, the equation \eqref{fulleq} can be rewritten as
\begin{equation}\label{fulleq1}
 I=M\widetilde{z}+\epsilon \widetilde{P}(\widetilde{z},\epsilon),
\end{equation}
where $\widetilde{P}(\widetilde{z},\epsilon)=\epsilon^{-1}P(\epsilon\widetilde{z})$. Since $P$ vanishes to order $2$,
we have that $\tilde P$ is an analytic function.

The formulation of the equation as \eqref{fulleq1} makes it
clear that we can use the implicit function theorem.
We will show that, for a set of $\nu$ of
full measure (see Proposition~\ref{dernonzero} in the following) we have that the matrix $M$ is invertible,
it follows from the implicit function theorem that there exists a solution $z=z(\epsilon)$ next to zero for the equation \eqref{fulleq}.

Note, however that the problem we have is not just to find solutions
$z$ of \eqref{fulleq}. Since the meaning of the components of
$z$ are squares of  modulus, we need that both of them are positive.
Of course, we could consider choosing the sign of $\epsilon$,
but it is non-trivial than choosing one sign of $\epsilon$.
We can ensure that both components of the solution have
a positive sign.

Hence,  we will show in Proposition~\ref{samesign}
 that the components of $M^{-1}I$ has the same sign. Precisely, when the components are positive, for $\epsilon>0$, we interpret the solutions $z$ for the equation \eqref{fulleq} as the absolute values of two complex numbers. When the components are negative, we get solutions $z$ for $\epsilon<0$.
We note that the condition is an explicit condition, on the vectors
in the kernel, was well as $\nu, \frak{m}$.

\smallskip

The remaining task is to give the formula of the matrix $M$ and prove it is invertible in a set of $\nu$ of  full measure and that the solutions of \eqref{fulleq} have both components positive for small values $\epsilon$.
\smallskip

By observation we find that when $d=2,$ the following two statements hold:
for $p,q,j=1,\cdots,4,$

$(S1):$ If $\exp\{\mathrm{i}k^px\},\exp\{\mathrm{i}k^qx\}\in Ker,$ then $\exp\{\mathrm{i}(k^p+k^q)\}$ is in the range space.

$(S2):$ If $\exp\{\mathrm{i}k^px\}, \exp\{\mathrm{i}k^qx\}, \exp\{\mathrm{i}k^jx\}\in Ker$ and $\exp\{\mathrm{i}(k^p+k^q+k^j)x\}\in Ker,$ then two of $k^p, k^q, k^j$ are opposite.
\smallskip

Consider  the range equation  \eqref{R}. It follows from the fact $\hv$ should be a quadratic function of $\ov,$ i.e., $\hv=O(\ov^2)$ that
the  equation  \eqref{R} becomes:
\begin{equation*}
(\LLo+(\mathfrak{m}-\mathfrak{m}_{0}))\hv=\ov^2+ O(\ov^3).
\end{equation*}
Since we assume $\mathfrak{m}-\mathfrak{m}_{0}=O(\epsilon), $ we have
\begin{equation*}
\LLo\hv=\ov^2+ O(\ov^3).
\end{equation*}
As a consequence,
$\hv=O(\ov^2)$ has  the form:
\begin{equation*}
\begin{split}
\hv&=\sum_{p,q=1}^{4}\frac{\alpha_{p}\alpha_{q}}{\Upsilon_{k^p+k^q}}
\exp\{\mathrm{i}(k^p+k^q)x\}+ O(|\alpha_p|+|\alpha_q|)^{3}\\
&=\frac{\alpha_{1}^2\exp\{2\mathrm{i}k^1x\}}{\Upsilon_{(2a,2b)}}+\frac{\alpha_{2}^2\exp\{2\mathrm{i}k^2x\}}{\Upsilon_{(2a,-2b)}}+\frac{\alpha_{3}^2\exp\{2\mathrm{i}k^3x\}}{\Upsilon_{(-2a,2b)}}+\frac{\alpha_{4}^2\exp\{2\mathrm{i}k^4x\}}{\Upsilon_{(-2a,-2b)}}\\
&\ \ \ + \frac{2\alpha_{1}\alpha_{2}\exp\{\mathrm{i}(k^1+k^2)x\}}{\Upsilon_{(2a,0)}}+\frac{2\alpha_{1}\alpha_{3}\exp\{\mathrm{i}(k^1+k^3)x\}}{\Upsilon_{(0,2b)}}\\
&\ \ \ +\frac{2\alpha_{2}\alpha_{4}\exp\{\mathrm{i}(k^2+k^4)x\}}{\Upsilon_{(0,-2b)}} +\frac{2\alpha_{3}\alpha_{4}\exp\{\mathrm{i}(k^3+k^4)x\}}{\Upsilon_{(-2a,0)}}+\frac{2(\alpha_{1}\alpha_{4}+\alpha_{2}\alpha_{3})}{\Upsilon_{(0,0)}},
\end{split}
\end{equation*}
which is well defined since $\Upsilon_{k^p+k^q}\neq0$ according to $(S1).$

\medskip

Consider the bifurcation eqaution \eqref{B}. Combining  $(S1)$ and $(S2),$ one has
\begin{equation*}\label{pro-2dim}
\begin{split}
 \epsilon\ov&=\Pi_{K}(\ov+\hv)^{2}\\
&=\Pi_{K}(2\ov\hv)+ O(|\alpha_{1}|+|\alpha_{2}|)^{4}\\
&=
\bigg(\frac{2\alpha_1^2\alpha_4}{\Upsilon_{(2a,2b)}}+\frac{4(\alpha_{1}\alpha_{4}+\alpha_{2}\alpha_{3})\alpha_{1}}{\Upsilon_{(0,0)}}+4\alpha_{1}\alpha_{2}\alpha_{3}(\frac{1}{\Upsilon_{(0,2b)}}+\frac{1}{\Upsilon_{(2a,0)}})\bigg)\exp\{\mathrm{i}k^1x\}
\\
&\ \ \ +\bigg(\frac{2\alpha_2^2\alpha_3}{\Upsilon_{(2a,-2b)}}+\frac{4(\alpha_{1}\alpha_{4}+\alpha_{2}\alpha_{3})\alpha_{2}}{\Upsilon_{(0,0)}}+4\alpha_{1}\alpha_{2}\alpha_{4}(\frac{1}{\Upsilon_{(0,-2b)}}+\frac{1}{\Upsilon_{(2a,0)}})\bigg)\exp\{\mathrm{i}k^2x\}\\
&\ \ \ +\bigg(\frac{2\alpha_3^2\alpha_2}{\Upsilon_{(-2a,2b)}}+\frac{4(\alpha_{1}\alpha_{4}+\alpha_{2}\alpha_{3})\alpha_{3}}{\Upsilon_{(0,0)}}+4\alpha_{1}\alpha_{3}\alpha_{4}(\frac{1}{\Upsilon_{(0,2b)}}+\frac{1}{\Upsilon_{(-2a,0)}})\bigg)\exp\{\mathrm{i}k^3x\}\\
&\ \ \ +\bigg(\frac{2\alpha_4^2\alpha_1}{\Upsilon_{(-2a,-2b)}}+\frac{4(\alpha_{1}\alpha_{4}+\alpha_{2}\alpha_{3})\alpha_{4}}{\Upsilon_{(0,0)}}+4\alpha_{2}\alpha_{3}\alpha_{4}(\frac{1}{\Upsilon_{(0,-2b)}}+\frac{1}{\Upsilon_{(-2a,0)}})\bigg)\exp\{\mathrm{i}k^4x\}\\
&\ \ \ + O(|\alpha_{1}|+|\alpha_{2}|)^{4}\\
&=\alpha_1\bigg(|\alpha_1|^2A+
|\alpha_2|^2B \bigg)\exp\{\mathrm{i}k^1x\}
+\alpha_2\bigg(|\alpha_2|^2A+
|\alpha_1|^2B\bigg)\exp\{\mathrm{i}k^2x\}\\
&\ \ \
+\alpha_3\bigg(|\alpha_3|^2A+
|\alpha_4|^2B\bigg)\exp\{\mathrm{i}k^3x\}
+\alpha_4\bigg(|\alpha_4|^2A+
|\alpha_3|^2B\bigg)\exp\{\mathrm{i}k^4x\}\\
&\ \ \ + O(|\alpha_{1}|+|\alpha_{2}|)^{4},
\end{split}
\end{equation*}
where
\begin{equation}
	\begin{split}
	&A=\frac{2}{\Upsilon_{(2a,2b)}}+\frac{4}{\Upsilon_{(0,0)}},\\
	&B=\frac{4}{\Upsilon_{(2a,0)}}+\frac{4}{\Upsilon_{(0,2b)}}+\frac{4}{\Upsilon_{(0,0)}}.
	\end{split}
\label{AB}
\end{equation}

Then, the linear part of the equation  \eqref{fulleq} for the non-zero $|\alpha_{1}|^2,\,|\alpha_{2}|^2$, is the following factorized equation;
\begin{equation*}
\left\{
\begin{array}{l}
\epsilon=|\alpha_1|^2A+
|\alpha_2|^2B,\\ \\
\epsilon=|\alpha_1|^2B+
|\alpha_2|^2A.
\end{array}
\right.
\end{equation*}
We denote by
 \begin{equation}\label{matrix}
 	\begin{split}
  M=\left(
  \begin{array}{l}
  A \ \ B\\
  B \ \  A
  \end{array}
  \right).
 	\end{split}
 \end{equation}
Indeed, $M$ is invertible in a  full measure set of $\nu$. See the following Proposition.
\begin{proposition}\label{dernonzero}
The determant of the matrix $M$ defined in \eqref{matrix} is different from zero for  a set of  $\nu$ of full measure.
\end{proposition}

\begin{proof}
	Since
	\[ \det(M)=A^2-B^2=(A+B)\cdot(A-B),\]	
 it suffices to consider  $A\pm B$. If we consider $A\pm B$ as a function of $\nu$, it  is  a rational function.

It is not difficult to compute the numerators of $A\pm B$ and
to check that they have a non-trivial term.   So, both
$A\pm B$ are non-trivial rational functions of $\nu$. Therefore,
they can vanish only on a set of $\nu$ of measure zero.  This set is
the set alluded to in the hypothesis of Theorem~\ref{thm:bifurcation}.

Note that the set of $\nu$ for which $A\pm B$ vanish depends on
$k$ and if we fix $k$ this is
the only set we need to exclude for this $k$. Since the set of $k$ is countable
we can exclude a set of $\nu$  for all the $k$.
\end{proof}

\begin{proposition}\label{samesign}
With the notations above, we have  that both components of $M^{-1} I$ have the same sign.
\end{proposition}
\begin{proof}
We have
	 \begin{equation*}
	\begin{split}
		M^{-1} I =\frac{1}{\det (M)}\left(
	\begin{array}{l}
	A \ \ -B\\
	-B \ \ \ A
	\end{array}
	\right)I=\frac{1}{A^2-B^2}\left(
	\begin{array}{l}
	A-B\\
	A-B
	\end{array}
	\right)=\frac{1}{A+B}\left(
	\begin{array}{l}
	1\\
    1
	\end{array}
	\right).
	\end{split}
	\end{equation*}

We recall that we have included in our assumptions that
the parameters $\nu$ are such that $A+B$ is not zero, so that the leading
term of the solution has a definitive sign.  Both components of
the leading solution
have the same sign (they are identical)
 and we can choose $\epsilon$ having the same sign with $A+B$ (this is the $\sigma$ included in Theorem~\ref{thm:bifurcation}), so that the solutions for the equation \eqref{fulleq} are
positive.

Once we know that the leading approximation is positive, the implicit
function theorem tells us we can choose a family $z$, which remains
positive for $\epsilon$ such that $|\epsilon|$ is small enough and
the sign of $\epsilon$ is $\sigma$.

	\end{proof}
\subsubsection{The bifurcation equation for $d =1$.}
The case of $d=1$ is much easier since  Proposition ~\ref{Balpha} and Propoaition~\ref{dernonzero} are easier to be proved for $d=1$.

It is obvious that the kernal space of the operator $\LLo$ is 2-dimensional, which implies
that the kernal space of the operator $\LLo$ can be represented as the following :
\begin{equation}\label{ker}
\begin{split}
Ker:=\{\ov_{\alpha}:\mathbb{T}_{\rho}\rightarrow\mathbb{C}\,\big|\, &\ov_{\alpha}(x)=\alpha_{1}\exp\{\mathrm{i}k^1x\}+\alpha_{2}\exp\{\mathrm{i}k^2x\},
\alpha_{1},\alpha_{2}\in\mathbb{C},\\
&and \,\, k^{j}\in\mathbb{Z}\ \  satisfying \,\, \Upsilon_{k^j}=0,\, j=1,2.\}
\end{split}
\end{equation}
Note that $k^{1}=-k^{2},$ then for real functions, one has $\alpha_{1}=\alpha_{2}^{*}.$

First, by the range equation \eqref{R} , $\hv=O(\ov^2)$ has the form of
\begin{equation}\label{range}
\hv=
\alpha_{1}^{2}\exp\{\mathrm{i}(2k^1x)\}\Upsilon_{2k^1}^{-1}+
\alpha_{2}^{2}\exp\{\mathrm{i}(2k^2x)\}\Upsilon_{2k^2}^{-1}+
2\alpha_{1}\alpha_{2}\Upsilon_{0}^{-1}+ O(|\alpha_{1}|+|\alpha_2|)^3,
\end{equation}
which is well defined since $\Upsilon_{2k^1}, \Upsilon_{2k^2}, \Upsilon_{0}\neq0.$

In order to obtain the coefficients $\alpha_1, \alpha_2,$ we concentrate on  the bifurcation equation  \eqref{B}.
Combing with \eqref{ker} and \eqref{range}, one has
\begin{equation*}
\begin{split}
\epsilon\ov&=\Pi_{K}(\ov+\hv)^{2}\\
&=\Pi_{K}(2\ov\hv)+ O(|\alpha_{1}|+|\alpha_{2}|)^{4}\\
&=
(\frac{2\alpha_1^2\alpha_2}{\Upsilon_{2k^{1}}}+\frac{4\alpha_1^2\alpha_2}{\Upsilon_0})\exp\{\mathrm{i}k^1x\}
+(\frac{2\alpha_1\alpha_2^2}{\Upsilon_{2k^{2}}}+\frac{4\alpha_1\alpha_2^2}{\Upsilon_0})\exp\{\mathrm{i}k^2x\}+ O(|\alpha_{1}|+|\alpha_{2}|)^{4}\\
&=2|\alpha_1|^2\alpha_1\bigg(\frac{1}{\Upsilon_{2k^1}}+\frac{2}{\Upsilon_0}\bigg)\exp\{\mathrm{i}k^1x\}+
2|\alpha_1|^2\alpha_2\bigg(\frac{1}{\Upsilon_{2k^2}}+\frac{2}{\Upsilon_0}\bigg)\exp\{\mathrm{i}k^2x\}\\
&\quad+ O(|\alpha_{1}|+|\alpha_{2}|)^{4}.
\end{split}
\end{equation*}

For nonzero $\alpha_{1}$,
\begin{equation}\label{M}
M=2(\frac{1}{\Upsilon_{2k^{1}}}+\frac{2}{\Upsilon_0})=\frac{5}{3m_0}> 0.
\end{equation}

Using Lemma~\ref{noeven}
and  noting that Proposition~\ref{Balpha} applies also
to the case $d=1$, we obtain that the bifurcation equation \eqref{fulleq} can be written
as
\[
\ep \alpha = \alpha P(| \alpha|^2)
\]
with $P$ an analytic function and $P(0) = 0$.

In the previous analysis, we have computed $P'(0)=M$ and,
in particular shown  that $P'(0) \ne 0$ by \eqref{M}. Hence, we can define a local
inverse for $P$ and the branches are given by
$ |\alpha|^2 = P^{-1}(\ep) = P'(0)^{-1} \ep + O(\ep^2)$.
 Note that, since $|\alpha|^2 \ge 0 $, we only obtain solutions for
$\ep$ with a fixed positive sign.

Notice that the fact that the bifurcation equations determines
only  $|\alpha| $ and not the phase is consistent with
Remark~\ref{symmetry}.

\subsubsection{Some remarks about the case $d \ge  3$}
Unfortunately, when the dimension is bigger,  the algebra
becomes more complicated.

Notably the factorization of the bifurcation equation does
not hold.
Note
\[
(a,b,c) = (a, b, -c) + (a, -b, c) + (-a,b,c)
\]
so that the bifurcation equation for $\alpha_1$ contain a
term which does not have a factor $\alpha_1$.  There are many other
examples.

\begin{remark}
	In 3 or higher dimensions, the statement $(S2)$ is not true and the only thing we can say, at the moment, about the bifurcation equations is that they have the form
	\[
	\epsilon \alpha_n=P_n(\alpha),
	\]
	where $P_n$ is a homogeneous polynomial of degree 3 or higer order with real coefficients.
	
	In dimension $d\geq3,$ the bifurcation equations include $2^d$ real variables ($2^{d-1}$ complex variables). The symmetry in Remark \ref{symmetry} shows that solutions have to be related in $d$-dimensional families. When $d\geq3,$ $2^{d-1}>d.$ So that the bifurcation equations is expected to give
more branches. Also the branches are not just charaterized by the
absolute values since there are more variables than phases
to adjust using Remark~\ref{symmetry}.  Of course, it is  possible
that there are other symmetries beyond the ones pointed out in
Remark~\ref{symmetry}.
\end{remark}

\section{Ill-posed Evolution equations}\label{illposedcase}

\subsection{A fixed point approach}
A very similar approach in Section \ref{nonresonant} can be applied to the problem of finding solutions for the nonlinear elliptic type evolution equations \eqref{evolution}.

By the analysis in Section \ref{mainidea}, the problem is equivalent to looking for solutions of the form $U(\theta,x): \mathbb{T}_{\rho}^{b}\times\mathbb{T}_{\rho}^{d}\rightarrow \mathbb{C}$ for \eqref{evolution1}:
\begin{equation}\label{pde}
\begin{split}
\mathcal{Q}_{\omega,\nu,\frak{m}}U=\epsilon \mathcal{N}(U),
\end{split}
\end{equation}
where
\begin{equation}\label{nonlinearitypde}
\begin{split}
\mathcal{N}(U)(\theta,x):=f(\theta,x,U(\theta,x),D_xU(\theta,x),D^2_xU(\theta,x)).
\end{split}
\end{equation}
Moreover, we assume the parameters $\omega, \nu,\frak{m}$ meet the hypothesis (H1) or (H2) mentioned in Section \ref{illfixed}.

As in the previous analysis, the key is to define an
appropriate space.

For $\rho\geq0, r, b, d\in \mathbb{Z}_{+},$
we define the following space of analytic functions $U(\theta,x)$ in $\mathbb{T}_{\rho}^{b+d}$ with finite norm:
\begin{equation*}
\begin{split}
H_*^{\rho,r}:
&=H^{\rho,r}_* (\mathbb{T}^{b+d}_{\rho})\\
&=
\bigg\{
U: \mathbb{T}_{\rho}^{b}\times\mathbb{T}_{\rho}^{d}\rightarrow\mathbb{C}
\,\,\bigg|\,\,
U(\theta,x)=\sum_{(l,k)\in\mathbb{Z}^{b}\times\mathbb{Z}^{d}}\hat{U}_{l,k}e^{\mathrm{i}(l\theta+kx)},\\
&\quad\quad
\|U\|_{\rho,r}^{2}=\sum_{(l,k)\in\mathbb{Z}^{b}\times\mathbb{Z}^{d}}|\hat{U}_{l,k}|^{2}e^{2\rho(|l|+|k|)}(1+|l|^{2})^r( 1 + |k|^{2})^{r}<\infty
\bigg\}.
\end{split}
\end{equation*}

\begin{remark}
It is natural to think of $H^{\rho, r}_*$
as a space of functions from $\torus^b_\rho$ into $H^{\rho, r}(\torus^d_\rho)$.
We think of $U(\omega t, \cdot)$ as a  quasi-periodic function in the
space $H^{\rho, r}$ of functions of $\torus^d_\rho$.  From this point of
view, it would have been natural to include different parameters
for the regularity in $\theta$ and the regularity in $x$, but
we have decided not to include it to avoid creating more complexity.

Note that the norm is equivalent to the norm
\[
  \|U\|_{\rho,r}^{2}=\sum_{(l,k)\in\mathbb{Z}^{b}\times\mathbb{Z}^{d}}|\hat{U}_{l,k}|^{2}e^{2\rho(|l|+|k|)}(1+|l|^{2} + |k|^{2})^{r}<\infty.
\]
\end{remark}

Before giving the main result of the evolution equation, we need to introduce the following lemma about the measure estimates of the parameter sets, which will produce resonance, corresponding to the hypotheses (H1) and (H2) respectively.

\begin{lemma}\label{mes}
Given fixed $\frak{m}>0$ and sufficiently small $\delta>0,$ we consider the following set of parameters $(\omega, \nu)$,
\begin{equation*}
\begin{split}
\tilde{I}=
\bigg\{
(\omega,\nu)\in[1,2]^{1+d}
\,\,\big|\,\,\exists\,\,l\in\mathbb{Z}^{1},\,\, k\in\mathbb{Z}^{d},\,\, such\,\, that\,
\left|-\omega^{2}l^{2}-\sum_{i=1}^{d}k_{i}^{2}\nu_{i}^{2}+\frak{m}\right|\leq\delta
\bigg\},
\end{split}
\end{equation*}
corresponding to the hypotheses $\mathrm{(H2)}$ in the Section ~\ref{illfixed} .
Then, the set have Lebesgue measure:
\[
\mathrm{mes}(\tilde{I})=O(\delta).
\]
\end{lemma}

In this case, we can regard the parameters set as the nonresonant elliptic case  where the dimension of $x$ increases by 1. Thus,
the Lemma~\ref{mes} can be obtained by adapting
slightly the proof of Lemma \ref{mesI},
we omit it.

\medskip
Using Proposition \ref{operatornorm}, we obtain the following estimates:
\begin{proposition}
If the parameters $\omega,\nu$ meet the hypotheses (H1) (or (H2)), then for all
$(\omega,\nu)\subset \mathbb{R}^b\times[1,2]^{d}$ (or $(\omega,\nu)\in[1,2]^{1+d}\setminus\tilde{I}$), we have
\begin{equation*}
\begin{array}{l}
\|\Q^{-1}\|_{H_*^{\rho,r}\rightarrow H_*^{\rho,r}}\leq\delta^{-1},\\
\|\Q^{-1}\|_{H_*^{\rho,r-2}\rightarrow H_*^{\rho,r}}<C,
\end{array}
\end{equation*}
where $C$ is a constant.
\end{proposition}

Now, we give the following result for the evolution equations:

\begin{theorem}
For fixed $\mathfrak{m}>0,$ any $0<\delta\ll1$, given $\rho\geq0, r-2>d/2,$
let $\mathbb{B}_s(0)\subset H_*^{\rho,r}$ be closed ball around the origin with the radius $s>0.$ Suppose that the parameters $\omega,\nu$ meet the hypotheses (H1) (or (H2)).

Assume that $\mathcal{N}$ defined in \eqref{nonlinearitypde} is Lipschitz from $\mathbb{B}_s(0)\subset H_*^{\rho,r}$ into $H_*^{\rho,r-2}.$

Then, there exists $\epsilon_{*}>0$ depending on $\nu,\frak{m},\delta,s,\Lip(\mathcal{N})$,
such that when $0<\epsilon<\epsilon_{*},$ for any
$(\omega,\nu)\in \mathbb{R}^{b}\times[1,2]^{d}$
(or
$(\omega, \nu)\in[1,2]^{1+d}\backslash\tilde{I}$,\, $\tilde{I}\subset[1,2]^{1+d}$ with
$\mathrm{mes}(\tilde{I})=O(\delta)$), the equation \eqref{pde} admits
a unique solution $u\in \mathbb{B}_s(0)$.
\label{illposedThm}
\end{theorem}

The proof of Theorem \ref{illposedThm} is very similar to Theorem \ref{analyticThm},
we omit it.

\section{Time-dependent center manifold approach}\label{center}

We notice that the method in Section \ref{illposedcase} can not solve
all the cases when the parameters set leads to the center direction.
Thus, we will introduce the center manifold theorem \cite{dlLR09,  CdlLR20}
 which is a powerful tool to analyse the evolution
equation.

These results construct a finite-dimensional quasi-periodic manifold
(with boundary)
inside a function space of solutions.

This quasi-periodic manifold in function space has the property that the
PDE restricted to the manifold is equivalent to an ODE in the manifold.
Therefore, the solutions of the ODE that do not reach the boundary
the solutions of the PDE  stay in the manifold for a short time.
Hence, to analyze the behaviour of the PDE, we can study  the behavior
of the finite-dimensional system given by the motion in this
manifold. The solutions of the finite-dimensional system will correspond
to solutions of the PDE.

Similar procedures (often called also
\emph{reduction principles}) have been used in PDE, including
ill-posed PDE.  Notably, in the case of elliptic PDE in
cylindrical domains \cite{KirchgassnerS79, Mielke91}.
Once the existence of invariant manifolds is established,
one can use standard methods of finite-dimensional dynamical
systems to establish a variety of solutions \cite{PolacikV17, PolacikV20}.
The case of time-dependent manifolds, which is the most  relevant for us
was developed in \cite{CdlLR20}.
The method of \cite{CdlLR20} gives information on the center manifold
and the dynamics on it. Then, any finite result of finite-dimensional
systems that gives computable  conditions for the existence of
an interesting solution, can be adapted to the PDE. The method presented
here gives expressions for the dynamics in the manifold given the
form of the PDE.  Imposing that the dynamics in the manifold satisfies
the conditions of the constructive theorems is ensured by explicit conditions
on the PDE. We will not give explicit examples of this rather standard
but long  calculations. Some interesting examples appear in \cite{HaragusI11}
.
Note that the invariant quasi-periodic manifold will be only a finite-differentiable function in the space, even if the space itself consists of
functions of analytic functions in space and, therefore, the solutions of
the PDE are analytic in space and time. Even if each of
the solutions are analytic, the  finite differentiability refers to
the way that these solutions are stacked together.

The strategy of  \cite{CdlLR20}, which we will implement in this section,
consists in deriving
a functional equation for the representation of  a time-dependent  locally  invariant manifold
as a graph, formulate an invariance equation  and reduce it into a fixed point problem.  It is
quite remarkable that the method applies even when the equation is ill-posed. Many standard
methods in invariant manifold theory such as the graph transform do not apply.

In this section, we will consider  \eqref{evolution} with the frequency
$\omega\in \mathbb{T}^{b},$ $b\in\mathbb{Z}_{+}$.
We will allow  that the  the forcing $f$ depends on $u, D_x u $ and
present a very very explicit proof of  Theorem~\ref{CM}.

As we will see
in Remark~\ref{morederivatives}  the methods of \cite{CdlLR20} allows
to deal with forcing terms that depend on higher (fractional) derivatives
but they cannot deal with nonlinearities depending on $D^2_x u$.  Since, for
possible applications it will be important to obtain explicit formulas, we
have decided to present full details in case simpler than another
one with optimal regularity.

More precisely, we  consider
\begin{equation}
u_{tt}+\sum_{i=1}^{d}\nu_{i}^{2}\frac{\partial^{2}}{\partial x_{i}^{2}}u+\frak{m}u=\epsilon f(\omega t,x,u,D_{x}u), \,\,x\in\mathbb{T}^{d},\,\,t\in\mathbb{R}.
\label{ill1}
\end{equation}
Setting $u_{t}=v,\,\,z=(u,v)^{\top},$ \eqref{ill1}
can be rewritten as the system
\begin{equation}
\left\{
\begin{array}{l}
\dot{\theta}=\omega\\
\dot{z}=\mathcal{A}z+\epsilon \mathcal{N}(\theta,z)
\end{array},
\right.
\label{illeq}
\end{equation}
where
\begin{equation}
 \mathcal{A}=
 \left(
  \begin{array}{cc}
  0&1\\
  -\sum\limits_{i=1}^{d}\nu_{i}^{2}\frac{\partial^{2}}{\partial x_{i}^{2}}-\frak{m}&0
  \end{array}
 \right),
 \quad\quad
 \mathcal{N}=
 \left(
   \begin{array}{c}
   0\\
   f(\omega t,x,u,D_{x}u)
   \end{array}
 \right).
\end{equation}
\subsection{Choice of spaces}
To construct the center manifold, we first need to choose the suitable Banach spaces which admit cut-off functions and such that the nonlinear operator is differentiable in them. (The paper \cite{CdlLR20} uses
the two space approach of \cite{Henry81} and obtains results with weaker regularity. See Remark~\ref{morederivatives}.

For \eqref{illeq},
we consider the analytic function $u$ in $H^{\rho,r}$ which is a Hilbert space and admits a cut-off function.
Then
\[
z=(u,v)^{\top}\in X:=H^{\rho,r}\times H^{\rho,r-1}.
\]
We will assume that $f$ is analytic.
By Banach algebra and composition properties, we have that for $r>d/2+1,$
\[
f(\omega t,x,u,D_{x}u)\,:\,\mathbb{T}^{b}\times\mathbb{T}^{d}\times H^{\rho,r}\times H^{\rho,r-1}\rightarrow H^{\rho,r-1}.
\]
Thus, we have that the nonlinearity $\mathcal{N}$ is bounded from $X$ to $X$.

\subsection{Analysis of the Linear term}
To identify the basis of the stable and unstable spaces, we will analyse the linear operator $\mathcal{A}.$

Let $\Lambda:=\{+1,-1\}$, one can check that
$\Phi_{k}^{\Lambda}(x)=(e^{\mathrm{i}kx}, \lambda_{k}^{\Lambda}e^{\mathrm{i}kx})^{\top}$ is the eigenvector of $\mathcal{A}$ belonging to the eigenvalue
\[
\mathrm{Spec}(\mathcal{A})=
\{\lambda_{k}^{\Lambda}\}_{k\in\mathbb{Z}^{d}}=\left\{\Lambda\bigg(\sum_{i=1}^{d}k_{i}^{2}\nu_{i}^{2}-\frak{m}\bigg)^{1/2}\right\}_{k\in\mathbb{Z}^{d}}.
\]
Thus, we take $\{\Phi_{k}^{\Lambda}(x)\}_{k\in\mathbb{Z}^{d}}$ as a basis of the space of $X.$

\begin{remark}
The operator $\mathcal{A}$ has discrete spectrum in $X$. Furthermore, we have:

1) The center spectrum of $\mathcal{A}$ consists of a finite number of  eigenvalues, since there is only a finite number of $k$ satisfying:
\[
\sum_{i=1}^{d}k_{i}^{2}\nu_{i}^{2}-\frak{m}\leq0,\,\,k\in \mathbb{Z}^{d},
\]
for fixed $\frak{m}>0.$

2) The hyperbolic spectrum is well separated from the center spectrum.
\label{eigenvalue}
\end{remark}

Thus, we know that the spectrum of the linear operator $\mathcal{A}$ satisfies the following Proposition.
\begin{proposition}\label{A}
For fixed $\frak{m}>0,$ there exist $\beta_{1}>\beta_{3}^{-}\geq0, \beta_{2}>\beta_{3}^{+}\geq0,$ and a splitting of spectrum of linear operator $\mathcal{A}$, i.e.,
\[
\mathrm{Spec}(\mathcal{A})=\sigma_{s}\cup\sigma_{c}\cup\sigma_{u},
\]
where
\begin{equation}
\begin{split}
&\sigma_{s}=\{\lambda_{k}^{\alpha}\,|\, Re\lambda_{k}^{\alpha}<-\beta_{1},\,k\in\mathbb{Z}^{d}\},\\
&\sigma_{u}=\{\lambda_{k}^{\alpha}\,|\, Re\lambda_{k}^{\alpha}>\beta_{2},\,k\in\mathbb{Z}^{d}\},\\
&\sigma_{c}=\{\lambda_{k}^{\alpha}\,|\, \beta_{3}^{-}\leq Re\lambda_{k}^{\alpha}<\beta_{3}^{+},\,k\in\mathbb{Z}^{d}\}.
\end{split}
\end{equation}
We note that $\sigma_{c}$ contains not only the center eigenvalues but also the eigenvalues
with slow stability/unstability.

As a conclusion, there is a decomposition
\begin{equation}
X=X_{s}\oplus X_{c}\oplus X_{u},
\label{decomposition}
\end{equation}
where $X_{\widetilde{\sigma}},$ $\widetilde{\sigma}=s,c,u,$ which are invariant for $\mathcal{A},$ i.e., $\mathcal{A}(D(\mathcal{A})\cap X_{\widetilde{\sigma}})\subset X_{\widetilde{\sigma}}.$
We denote by $\Pi_{\widetilde{\sigma}}$ the projection operator over $X_{\widetilde{\sigma}},$ which is bounded in $X.$
\end{proposition}
\begin{proof}
We notice that the eigenvalues of $\mathcal{A}$ are discrete and $\lambda_{k}\rightarrow\infty$ when $k\rightarrow\infty.$ Therefore, we can choose appropriate $\beta_{1},\beta_{2},\beta_{3}^{+},\beta_{3}^{-},$ which can split the spectrum of $\mathcal{A}$ into $\sigma_{s},\sigma_{c},\sigma_{u},$ such that $\sigma_{s},\sigma_{c},\sigma_{u}$ are disjoint and cover all the eigenvalues. The existence of the decomposition is the point of the spectral theorem.
\end{proof}
In the dynamical systems theory, the conclusion of Proposition \ref{A} is described as $\mathcal{A}$ has a trichotomy
for the generator of the evolution. We note, however that the operator  $\mathcal{A}$ does
not generate an evolution. As we detail in Lemma~\ref{semigroups}, it generates semigroups in the future
or  in the past in subspaces.

\begin{lemma}\label{semigroups}

Denote by $\mathcal{A}_s, \mathcal{A}_u, \mathcal{A}_c$ the restrictions of $\mathcal{A}$
to $X_s, X_u, X_c$ respectively.

Then, we can define the following (semi)groups
\[
\begin{split}
& \{\mathcal{A}^{s}(t) \equiv e^{t \mathcal{A}^s}\}_{t \ge 0}  \\
& \{\mathcal{A}^{u}(t) \equiv e^{t \mathcal{A}^u}\}_{t \le 0}  \\
& \{\mathcal{A}^{c}(t) \equiv e^{t \mathcal{A}^c}\}_{t \in \mathbb{R}}  \\
\end{split}
\]
defined in the spaces $X_s,X_u, X_c$ respectively.

Moreover , the following estimates holds:
\begin{equation}
\begin{split}
&\|\mathcal{A}^{s}(t)\|_{X_s,X}\leq e^{-\beta_{1}t},\,\,t>0;\\
&\|\mathcal{A}^{u}(t)\|_{X_u,X}\leq e^{-\beta_{2}|t|},\,\,t<0;\\
&\|\mathcal{A}^{c}(t)\|_{X_c,X}\leq e^{\beta_{3}^{-}|t|},\,\,t\leq0;\\
&\|\mathcal{A}^{c}(t)\|_{X_c,X}\leq e^{\beta_{3}^{+}|t|},\,\,t\geq0.
\end{split}
\label{rate}
\end{equation}
\end{lemma}
\begin{proof}
Suppose $z\in X$ has the following Fourier expansion
\[
z=\sum_{k\in\mathbb{Z}^{d},\Lambda \in \{-1,1\} }\hat{z}_{k}^{\Lambda} \Phi_{k}^{\Lambda}(x)
\]
with norm in $X$ (which is equivalent to the norm in $H^{\rho,r}\times H^{\rho,r-1}$)
\[
\|z\|_{X}^{2}=\sum_{k\in\mathbb{Z}^{d},\Lambda \in \{-1,1\}}
 |\hat{z}_{k}^{\Lambda}|^{2}e^{2\rho|k|}(1+|k|^{2})^{r}.
\]
Then, for $z\in X^{s},$
\[
\mathcal{A}^{s}(t)z=\sum_{k\in\mathbb{Z}^{d},\Lambda \in \{-1,1\} }
e^{\lambda_{k}^{\Lambda}t}\hat{z}_{k}^{\Lambda}\Phi_{k}^{\Lambda}(x),\,\,\,t>0,
\]
we have
\begin{equation*}
\begin{split}
\|\mathcal{A}^{s}(t)z\|_{X}^{2}&=
\sum_{k\in\mathbb{Z}^{d},  \Lambda\in \{-1,1\}\atop \lambda_{k}^{\Lambda} \in\sigma_{s} }
|e^{\lambda_{k}^{\Lambda}t}\hat{z}_{k}^{\Lambda}|^{2}e^{2\rho|k|}(1+|k|^{2})^{r}\\
&\leq e^{-2\beta_{1}t}
\sum_{k\in\mathbb{Z}^{d},  \Lambda\in\{-1,1\}
	\atop \lambda_{k}^{\Lambda} \in\sigma_{s} }
|\hat{z}_{k}^{\Lambda}|^{2}e^{2\rho|k|}(1+|k|^{2})^{r}\\
&=e^{-2\beta_{1}t}\|z\|_{X}^{2}.
\end{split}
\end{equation*}
Therefore,
\[
\|\mathcal{A}^{s}(t)\|_{X_s,X}\leq e^{-\beta_{1}t},\,\,\,t>0.
\]
Similarly, the remaining three inequalities hold.
\end{proof}

\begin{remark}
Since the center space $X_{c}$ is finite-dimensional,  $X_{c}$ admits $C^{r}$ cut-off function.
\end{remark}

Since we only construct the evolutions of the equations with sufficiently small perturbations, we can not consider the equation \eqref{illeq} directly. We have to introduce the ``prepared equation'' (\cite{Lan73}).

For any $z\in X,$ $\Pi_{\sigma}z=z_{\sigma}.$ We consider the following prepared equation of \eqref{illeq}:
\begin{equation}
\begin{split}
\left\{
\begin{array}{l}
\displaystyle\frac{d\theta}{dt}=\omega\\ \\
\displaystyle\frac{dz_{s}}{dt}=\mathcal{A}_{s}z_{s}+\epsilon \mathcal{N}_{s}(\theta,x,z_{s},z_{c},z_{u})\\ \\
\displaystyle\frac{dz_{c}}{dt}=\mathcal{A}_{c}z_{c}+\epsilon \varphi(z_{c}) \mathcal{N}_{c}(\theta,x,z_{s},z_{c},z_{u})\\ \\
\displaystyle\frac{dz_{u}}{dt}=\mathcal{A}_{u}z_{u}+\epsilon \mathcal{N}_{u}(\theta,x,z_{s},z_{c},z_{u})
\end{array}
\right.
\end{split},
\label{prepared}
\end{equation}
where $\varphi\,:\, X\rightarrow\mathbb{R}$ is a $C^{r}$ cut-off function such that it is identically 1 in the ball of radius $1/2$ centered at the origin and 0 identically outside the ball of radius 1, where $C^{r}(X,Y)$ is a set of functions from $X$ to $Y$ that have continuous derivatives of order less than or equal to $r.$
Then, $\varphi(z_{c})\mathcal{N}_{c}(\theta,x,z_{s},z_{c},z_{u})$ is a uniformly $C^{r}$ function, we can also arrange that the $C^{r}$ norm of the $\epsilon$ is as small as we needed. We obtain the flow on $X_{c}$ and denote it by $J^{w}_{t}(z_{c}(0)).$
Denote by $\widetilde{\mathcal{N}}$ the nonlinearity
\[
\widetilde{\mathcal{N}}=\epsilon \bigg( \mathcal{N}_{s}(\theta,x,z_{s},z_{c},z_{u}),\,\, \varphi(z_{c}) \mathcal{N}_{c}(\theta,x,z_{s},z_{c},z_{u}),\,\, \mathcal{N}_{u}(\theta,x,z_{s},z_{c},z_{u})\bigg).
\]

Our goal is to find a function $w\,:\,\mathbb{T}^{b+d}\times X_{c}\rightarrow X_{s}\oplus X_{u},$ and verify the graph of $w$ which is denoted by
\[
\mathcal{W}=\big\{
(\Theta, w_{s}(\Theta,x,J^{w}_{t}(z_{c}(0))), J^{w}_{t}(z_{c}(0)), w_{u}(\Theta,x,J^{w}_{t}(z_{c}(0))))
\big\}
\]
is invariant under \eqref{prepared}.

In conclusion, we have checked that the system \eqref{prepared} satisfies the following hypothesis:

H1) The decomposition \eqref{decomposition} of the space $X$ is invariant under $\mathcal{A},$ $\Pi_{\widetilde{\sigma}}, \widetilde{\sigma}=s,c,u$ is bounded in $X.$

H2) The operator $\mathcal{A}$ generates semi-groups $\mathcal{A}^{s,c,u}(t),$ with the quantitative assumption \eqref{rate} on the contraction rates.

H3) The nonlinearity $\widetilde{\mathcal{N}}\,:\,\mathbb{T}^{b+d}\times X\rightarrow\mathbb{T}^{b+d}\times X$ is $C^{r}$ and $\|\widetilde{\mathcal{N}}\|_{C^{r}}$ is sufficiently small.

We will give a detailed and explicit proof  the following result, Theorem~\ref{CM} which is a
particular case of Theorem 3.1 of \cite{CdlLR20}.

To obtain Theorem~\ref{CM} from
Theorem 3.1 of  \cite{CdlLR20}, it suffices to take the two spaces $X,Y$ used in
\cite{CdlLR20} to be equal to $X$. We will present a detailed proof of
Theorem~\ref{CM}. In Remark~\ref{morederivatives} we will discuss the results
that are obtained using the full force of Theorem 3.1 of \cite{CdlLR20}.

\begin{theorem}\label{CM}
Assume in the space $X$, the linear operator $\mathcal{A}$ and the nonlinearity $\widetilde{\mathcal{N}}$ satisfy the assumptions H1),H2),H3) respectively.

Then, there exists a $C^{r-1+Lip}$ function $w\,:\,\mathbb{T}^{b+d}\times X_{c}\rightarrow X_{s}\oplus X_{u},$ and $\mathcal{W}$, the graph of $w$, is globally invariant by \eqref{prepared}. Furthermore, $\mathcal{W}$ is $C^{r-1+Lip}$ locally invariant by \eqref{illeq}.
\end{theorem}

\begin{remark}
Note that even if $f$ is $C^{\omega}$, the cut-off is only $C^{r}.$ The center manifold obtained will be invariant for the cut-off equations but only locally invariant for the original equation.

We are going to consider the equation as an evolution in spaces of analytic functions,
so that all the solutions of the PDE we consider, will be analytic in the space variable.
As a consequence, they will be also analytic in time. Nevertheless, in spite of the fact that the
solutions are analytic, the center manifold we construct will be only $C^r$ for a finite $r$.
Note however that it is a $C^r$ manifold in a space of analytic functions (of the space variable).

Even if for every $r$, we can find a $C^r$ manifold, it may be impossible to find a $C^\infty$ manifold.
This is because to increase the $r$ we may need to have a stronger cut-off in the preparation so that
we cannot take the limit. There are well known examples of this phenomenon even in finite-dimensional,
polynomial ODE's \cite{Lan73}.
\end{remark}

\subsection{The proof of Theorem \ref{CM}}
The proof of Theorem \ref{CM} is based on the contraction principle. We consider the center direction in \eqref{prepared} and have the following evolution equation for $z_{c}$:
\begin{equation*}
\begin{split}
\left\{
\begin{array}{l}
\displaystyle\frac{d\theta}{dt}=\omega\\ \\
\displaystyle\frac{dz_{c}}{dt}=\mathcal{A}_{c}z_{c}+\widetilde{\mathcal{N}}_{c}(\theta,x,z_{s},z_{c},z_{u})
\end{array}
\right.
\end{split}
\end{equation*}
with initial value $(\theta_{0}, z_{c}(0)).$
By the Duhamel principle, we obtain the solution of this equation:
\begin{equation*}
\Theta=\theta_{0}+\omega t,
\end{equation*}
\begin{equation*}
J_{t}^{w}(z_{c}(0))=e^{\mathcal{A}_{c}t}z_{c}(0)+\int_{0}^{t}e^{\mathcal{A}_{c}(t-\tau)}\widetilde{\mathcal{N}}_{c}(\Theta(\tau), x, J_{\tau}^{w}(z_{c}(0)),w(\Theta(\tau),J_{\tau}^{w}(z_{c}(0))))d\tau.
\label{c}
\end{equation*}
Moreover, one has
\begin{equation*}
z_{s}(t)=e^{\mathcal{A}_{s}t}z_{s}(0)+\int_{0}^{t}e^{\mathcal{A}_{s}(t-\tau)}\widetilde{\mathcal{N}}_{s}(\Theta(\tau), x, J_{\tau}^{w}(z_{c}(0)),w(\Theta(\tau),J_{\tau}^{w}(z_{c}(0))))d\tau, \,\,t\geq0,
\end{equation*}
\begin{equation*}
z_{u}(t)=e^{\mathcal{A}_{u}t}z_{u}(0)+\int_{0}^{t}e^{\mathcal{A}_{u}(t-\tau)}\widetilde{\mathcal{N}}_{u}(\Theta(\tau), x, J_{\tau}^{w}(z_{c}(0)),w(\Theta(\tau),J_{\tau}^{w}(z_{c}(0))))d\tau, \,\,t\leq0.
\end{equation*}
\medskip

To verify the graph of $w$ is invariant, we denote by
\[(z_{s},z_{u})=(w_{s}(\Theta,x,J^{w}_{t}(z_{c}(0))), w_{u}(\Theta,x,J^{w}_{t}(z_{c}(0)))).\]
From Duhamel principle and $(\Theta(t), J_{t}^{w}(z_{c}(0)))$ is invertible, do
some variable transformation on $t$, we have, when $t\rightarrow \mp\infty,$
\begin{equation}
w_{s}(\theta_{0},z_{c}(0))=
\int_{-\infty}^{0}e^{-\tau\mathcal{A}_{s}}\widetilde{\mathcal{N}}_{s}(\Theta(\tau), x, J_{\tau}^{w}(z_{c}(0)),w(\Theta(\tau),J_{\tau}^{w}(z_{c}(0))))d\tau,
\label{s}
\end{equation}
\begin{equation}
w_{u}(\theta_{0},z_{c}(0))=
-\int^{\infty}_{0}e^{-\tau\mathcal{A}_{u}}\widetilde{\mathcal{N}}_{u}(\Theta(\tau), x, J_{\tau}^{w}(z_{c}(0)),w(\Theta(\tau),J_{\tau}^{w}(z_{c}(0))))d\tau.
\label{u}
\end{equation}
We denote by $\mathcal{T}_{c}, \mathcal{T}_{s}, \mathcal{T}_{u}$ the RHS of equations \eqref{c},\eqref{s},\eqref{u} respectively.
Then we obtain a fixed point equation
\[
(J_{t}^{w}(z_{c}(0)),w_{s}(\theta_{0},z_{c}(0)),w_{u}(\theta_{0},z_{c}(0)))\equiv \mathcal{T}=(\mathcal{T}_{c}, \mathcal{T}_{s}, \mathcal{T}_{u}).
\]

The rest of the work is to construct solutions of \eqref{s} and \eqref{u}.

There are a fixed point of the operators that to
$w_{s}(\theta_{0},z_{c}(0))$ and $w_{u}(\theta_{0},z_{c}(0))$
associate the RHS of \eqref{s} and \eqref{u}, respectively. The proof
of the existence of the fixed point is done in great detail in
\cite{dlLR09,CdlLR20}. Both of them are based on a method from
\cite{Lan73}. The basic idea is to show
that there is a $C^{r+1}$ ball that gets mapped onto itself
by the operator (this is obtained  using the
estimates on composition of $C^{r+1}$ functions, the estimates on
derivatives of solutions of a an ODE. and the different rates).
The second step is to prove that this operator is a contraction in
a $C^0$ norm (This is done by applying systematically adding and
subtracting so that only one term is modified at the time.
The most difficult step is estimating the change of the solutions
of the ODE when the coefficients are changed).
By studying the properties of the solution, it is also shown
that the solutions of the fixed point problem are a solution of
the original problem.

Similar equations appear in the study of center manifolds.  Note that the
nonlinear perturbations we have considered are differentiable and
that the linear parts generate reasonable evolutions. So that there
is not much difference between the finite-dimensional proofs and
the proof needed. For a treatment of a similar problem, we refer to
\cite{Mielke91}. The paper \cite{dlLR09} deals with a more general
situation.

\begin{remark}\label{morederivatives}
The proof of \cite{CdlLR20} can deal with forcing nonlinearities
that are more singular than first derivatives.
The Theorem $3.1$ of \cite{CdlLR20}
applies to problems
\[
u_{tt} = u_{xx} + \F(\omega t, u),
\]
where $F$ is a differentiable  functional from the spaces indicated.
\[
\F:\torus^d \times H^{\rho, m} \rightarrow   H^{\rho, m -2 + \delta}, \quad \delta > 0.
\]
The very interesting case $\delta = 0 $ is not covered by the results.

The method of \cite{CdlLR20} goes through equations \eqref{s}, \eqref{u}
and it also uses   the strategy of proving propagated bounds and $C^0$
contraction.    The analysis, however, is more careful and takes advantage
-- following \cite{Henry81} of the fact that the operator $\mathcal{A}^s(t)$,
$\mathcal{A}^u(t)$ are smoothing. They are bounded operators  from
  $H^{\rho,m-2+\delta} \times H(\rho,m -3+\delta) $
to   $H^{\rho,m} \times H(\rho,m -1)$ and the bounds are integrable.
\end{remark}

\appendix

\section{Some Properties of $H^{\rho,r}(\mathbb{T}_{\rho}^{d})$}
\label{sec:properties}

In this section, we collect a few
lemmas about the properties on $H^{\rho,r},$ which play a crucial role in the proof. Similar contnents have appeared in other papers.
We note that Lemma~\ref{Banachalgebra}  assumes only
$r > d/2$ whereas in previous papers it was assumed $r > d$.
This  leads to similar improvements in the previous papers
\cite{CallejaCL13, CCCdlL17,WangL20}.

\medskip
A small observation that can  be found in the previous papers is
that the $H^{\rho, r}$ norm is equivalent to the $L^2$ norm of
derivatives up to order $r$.  The derivatives can be taken
to be either real derivatives or complex derivatives. That is,
\[
  \begin{split}
    \| u \|_{\rho, r} &\approx \|u\|_{L^2(\torus^d_\rho)} +
    \|D u\|_{L^2(\torus^d_\rho)} + \cdots   \| D^r u \|_{L^2(\torus^d_\rho)} \\
    &\approx
    \|(1- \Delta)^{r/2}  u \|_{L^2(\torus^d_\rho)} \\
    &\approx\| (1- \bar{\partial}_z \cdot \partial_z)^{r/2} u \|_{L^2(\torus^d_\rho)}.
  \end{split}
\]

\medskip

The following result is elementary but crucial:
\begin{proposition}\label{Linftybounds}
  Assume that $r > d/2$, then
  \begin{equation*}\label{eq:Linftybounds}
	  \sup_{z  \in \torus^d_{\rho}} |u(z)| \le
    C_{d,r} \|u\|_{\rho,r}
  \end{equation*}
  for some constant $ C_{d,r}$ depending on $d,\,r$.
\end{proposition}

  \begin{proof}
    Using triangle and Cauchy-Schwartz inequalities, we have:
\[
      \begin{split}
	      \sup_{z  \in \torus^d_{\rho}} |u(z)| &\le \sum_k
       | \hu_k| e^{\rho |k|}
        = \sum_k
        | u_k| e^{\rho |k|} (1 +|k|^2)^{r/2} (1+ |k|^2)^{-r/2} \\
        & \le \left( \sum_k | u_k|^2 e^{2 \rho |k|} \cdot
	      (1 +|k|^2)^{r} \right)^{1/2}
          \left(  \sum_k (1+ |k|^2)^{-r} \right)^{1/2} \\
         & = \|u\|_{\rho,r} C_{d,r}.
        \end{split}
\]
 \end{proof}
        Notice that this inequality is better than the Sobolev inequality
        if we considered $\torus^d_\rho$ as a $2d$ real manifold and
        $H^{\rho, r}$ as a closed space of the (real) Sobolev space
        $H^r( \torus^d_\rho)$.  Applying the real Sobolev embedding
        -- as was done in \cite{CallejaCL13} --- requires
        $r > d$.

        The reason is that, even if $\torus^d_\rho$ is a
        $2d$ dimensional real manifold, due to the maximum principle
        for analytic functions, the sizes of the functions in $H^{\rho,r}$
        are controlled by the $H^r$ norm to the restriction to
        of the functions to the $d$ dimensional manifolds given
        by $\mathrm{Im}(z_i) = \pm \rho$.  There are $2^d $ components
        each of which is a real $d$ dimensional torus.

        As we will see immediately, similar results appear in
        the Banach algebra properties.

      Note that we also get improved Sobolev embedding
      theorems. If $ r = d/2 + \lambda$,
      we obtain
      \begin{equation}\label{embedding}
        \| u \|_{C^\lambda(\torus^d_\rho)} \le C \|u\|_{\rho, r}.
      \end{equation}
      Of course, in \eqref{embedding}, the regularity in the interior is
      not an issue (the functions are analytic) but we obtain quantitative
      bounds.

\begin{lemma}Banach algebra properties:
\label{Banachalgebra}

\rm{(1)} Sobolev case: Let $\rho=0, r>d/2.$ Then there exists a positive constant $C_{r,d}$ depending on $r,\,d$, so that for any $u_{1},\,u_{2}\in H^{r}(\mathbb{T}^{d},\mathbb{R}),$ the product $u_{1}\cdot u_{2}$ is in $H^{r}(\mathbb{T}^{d},\mathbb{R}),$ and
\[
\|u_{1}\cdot u_{2}\|_{r}\leq C_{r,d}\|u_{1}\|_{r}\|u_{2}\|_{r}.
\]

\rm{(2)} Analytic case: Let $\rho>0, r>d/2.$ Then there exists a positive constant $C_{\rho,r,d}$ depending on $\rho, r, d,$ so that for any $u_{1}, u_{2}\in H^{\rho,r}(\mathbb{T}_{\rho}^{d},\mathbb{C}),$ the product $u_{1}\cdot u_{2}$ is in $H^{\rho,r}(\mathbb{T}^{d}_{\rho},\mathbb{C}),$ and
$$
\|u_{1}\cdot u_{2}\|_{\rho,r}\leq C_{\rho,r,d}\|u_{1}\|_{\rho,r}\|u_{2}\|_{\rho,r}.
$$
\end{lemma}

\begin{proof}
Part (1) is the classic result, see \cite{Tay11b,AF03}. We prove the part (2).

Denote the set $S:=\{\varsigma\,|\,\varsigma=\{1,-1\}^{d}\},$ where $\varsigma$ is a $d$-dimensional vector, and the $i$-th component is $1$ or $-1$, $i=1,\cdots,d.$
We choose each component of $-\varsigma^{*}$ has the same sign with $k$, then
\[
e^{-2k\varsigma\rho}\leq e^{2|k|\rho}=e^{-2k\varsigma^{*}\rho}\leq\sum_{\varsigma\in S}e^{-2k\varsigma \rho},\,\,\,for\,\,k\in\mathbb{Z}^{d}.
\]
For any Fourier coefficient $\hu_{k}$ of $u(x)$, we have
\[
|\hu_{k}|^{2}e^{-2k\varsigma\rho}
\leq|\hu_{k}|^{2}e^{2|k|\rho}
\leq\sum_{\varsigma\in S}|\hu_{k}|^{2}e^{-2k\varsigma\rho}.
\]
Multiplying the three sides of the above inequality by $(1+|k|^{2})^{r}$ and summing in $k$, we have
\begin{equation*}
\begin{split}
\sum_{k\in\mathbb{Z}^{d}}|\hu_{k}|^{2}e^{-2k\varsigma\rho}(1+|k|^{2})^{r}
&\leq\sum_{k\in\mathbb{Z}^{d}}|\hu_{k}|^{2}e^{2|k|\rho}(1+|k|^{2})^{r}\\
&\leq\sum_{\varsigma\in S}\sum_{k\in\mathbb{Z}^{d}}|\hu_{k}|^{2}e^{-2k\varsigma\rho}(1+|k|^{2})^{r}.
\end{split}
\label{equiv}
\end{equation*}
If we define the norm $\|\cdot\|_{\varsigma,\rho,r}$ as:
\begin{equation*}
\|u\|_{\varsigma,\rho,r}^{2}:=\sum_{k\in\mathbb{Z}^{d}}|\hu_{k}|^{2}e^{-2k\varsigma\rho}(1+|k|^{2})^{r},
\end{equation*}
we obtain that the norm $\|u\|_{\rho,r}$ is equivalent to $(\sum_{\varsigma\in S}\|u\|_{\varsigma,\rho,r}^{2})^{1/2}.$

First, we verify that $\|\cdot\|_{\varsigma,\rho,r}$ is a Banach algebra in the $d$-dimensional manifold:
$\{\theta\,|\,\mathrm{Re}(\theta)\in\mathbb{T}^{d}, \,\,\mathrm{Im}(\theta)=\varsigma\rho\}.$
Denote $a=\mathrm{Re}(\theta),$ and the function $\Gamma_{\varsigma}u\,:\,\mathbb{T}^{d}\rightarrow \mathbb{C},$
\[
(\Gamma_{\varsigma}u)(a)=u(a+\mathrm{i}\varsigma\rho).
\]
One has $\Gamma_{\varsigma}(u_{1}u_{2})(a)
=[\Gamma_{\varsigma}(u_{1})\Gamma_{\varsigma}(u_{2})](a).$
Thus, when $r>d/2,$ we have
\begin{equation*}
\|u_{1}u_{2}\|_{\varsigma,\rho,r}
=\|\Gamma_{\varsigma}(u_{1}u_{2})\|_{r}
\leq\|\Gamma_{\varsigma}(u_{1})\|_{r}\|\Gamma_{\varsigma}(u_{2})\|_{r}
=\|u_{1}\|_{\varsigma,\rho,r}\|u_{2}\|_{\varsigma,\rho,r}.
\end{equation*}

Then,
\begin{equation*}
\begin{split}
\bigg(\sum_{\varsigma\in S}\|u_{1}u_{2}\|_{\varsigma,\rho,r}^{2}\bigg)^{\frac{1}{2}}
&\leq\bigg(\sum_{\varsigma\in S}(\|u_{1}\|_{\varsigma,\rho,r}^{2}\cdot\|u_{2}\|_{\varsigma,\rho,r}^{2})\bigg)^{\frac{1}{2}}\\
&\leq\bigg(\sum_{\varsigma\in S}\|u_{1}\|_{\varsigma,\rho,r}^{2}\cdot\sum_{\varsigma\in S}\|u_{2}\|_{\varsigma,\rho,r}^{2}\bigg)^{\frac{1}{2}}\\
&\leq\bigg(\sum_{\varsigma\in S}\|u_{1}\|_{\varsigma,\rho,r}^{2}\bigg)^{\frac{1}{2}}\cdot
\bigg(\sum_{\varsigma\in S}\|u_{2}\|_{\varsigma,\rho,r}^{2}\bigg)^{\frac{1}{2}}.
\end{split}
\end{equation*}
That is to say, when $r>d/2,$ $(\sum_{\varsigma\in S}\|u\|_{\varsigma,\rho,r})^{1/2}$ is a Banach algebra.
By the equivalence of norms, we have completed the proof.
\end{proof}

Note that having Proposition~\ref{Linftybounds}, we could
have followed also the standard proof using the Leibnitz formula.

\medskip

For the purposes of this paper, the main issue is the study of
the operator given by composition on the left.

Many other composition properties
can be found in \cite{CCCdlL17} the Proposition 3.9 in \cite{Tay11b} for details. For more results, one can also refer to \cite{AZ90,IKT13,Mar74,RS96}.

\begin{lemma}\label{compos1} Composition properties:

(1) Sobolev case: Let $f\in C^{r}(\mathbb{R}^{n},\mathbb{R}^{n})$ and assume that $f(0)=0.$ Then, for $u\in H^{r}(\mathbb{T}^{d},\mathbb{R}^{n})\bigcap L^{\infty}(\mathbb{T}^{d},\mathbb{R}^{n}),$ we have
\[
\|f(u)\|_{r}\leq C_{r}(\|u\|_{L^{\infty}})(1+\|u\|_{r}),
\]
where $C_{r}:=C_{r}(\eta)=\sup_{|x|\leq\eta,\,\alpha\leq r}|D^{\alpha}f(x)|.$
Particularly, when $r>d/2,$ if $f\in C^{r+2}$ and $u,v,u+v\in H^{r},$ then
\begin{equation}
\begin{split}
\|f\circ(u+v)-f\circ u-Df\circ u\cdot v\|_{r}\leq
C_{r,d}(\|u\|_{L^{\infty}})(1+\|u\|_{r})\|f\|_{C^{r+2}}\|v\|_{r}^{2},
\end{split}
\label{composf}
\end{equation}
for some $C_{r,d}>0$ depending on the norm of $u.$

(2) Analytic case: Let $f:\,B\rightarrow \mathbb{C}^{n}$ with $B$ being an open ball around the origin in $\mathbb{C}^{n}$ and assume that $f$ is analytic in $B.$ Then, for $u\in H^{\rho,r}(\mathbb{T}^{d}_{\rho},\mathbb{C}^{n})\bigcap L^{\infty}(\mathbb{T}_{\rho}^{d},\mathbb{C}^{n})$ with $u(\mathbb{T}^{d}_{\rho})\subset B,$ we have
\begin{equation}\label{composition1}
\|f(u)\|_{\rho,r}\leq C_{\rho,r}(\|u\|_{L^{\infty}})(1+\|u\|_{\rho,r}).
\end{equation}
In the case of $r>d/2,$ we have
\begin{equation} \label{derivative}
\begin{split}
\|f\circ(u+v)-f\circ u-Df\circ u\cdot v\|_{\rho,r}\leq C_{\rho,r,d}(\|u\|_{L^{\infty}})(1+\|u\|_{\rho,r})\|v\|_{\rho,r}^{2}.
\end{split}
\end{equation}

As a corollary of \eqref{derivative}, we obtain that, under the hypotheses
of the Lemma, the operator $u\rightarrow f\circ u$ is differentiable.

Since the Hilbert space $H^{\rho, r}$ is a complex space, and the differentiability is in the complex sense, we conclude that the operator $u\rightarrow f(u)$ is
analytic.
\end{lemma}

\begin{proof}

The finite-differentiable case of
Lemma~\ref{compos1}  is a well known consequence
 of Gagliardo-Nirenberg-Moser composition estimates, for specific proof see Proposition 3.9 in \cite{Tay11b}.

 Here we give the proof of \eqref{composition1} and  \eqref{composf}.
 We notice that if $u$ is bounded (in particular if $r > d/2$ by
 Lemma~\ref{Linftybounds}) and that the range of $u$ is in the domain
 of $f$, we have that, by the chain rule, $f\circ u$ is complex
 differentiable in $\torus^d_\rho$ and that $f\circ u, (Df)\circ u$
 are bounded.

 Since $D(f\circ u) = Df\circ u Du$,
 we obtain, computing $\int_{T^d_\rho} |Df\circ u  Du|^2 $,
 that $f\circ u \in H^{\rho, 1}$ if
 $u \in H^{\rho,1}$.  To get the result for arbitrary $r$, we
 can use the Faa-Di-Bruno formula for higher derivatives and use
 the bounds we already have from the previous stages.

Let $u, v\in H^{r},$ $\xi=u+\zeta v$ for some $\zeta\in [0,1],$ $\xi$ is in the domain of $f.$
By the fundamental theorem of calculus, we have
\begin{equation*}
\begin{split}
f(u+v)&=f(u)+\int_{0}^{1}Df(\xi)\cdot v d\zeta\\
&=f(u)+Df(u)\cdot v+\int_{0}^{1}\int_{0}^{\zeta}D^{2}f(u+\zeta tv)v^{2}dtd\zeta.
\end{split}
\end{equation*}

The fact that
the differentiable functions of a complex Banach space are
analytic is proved in \cite[Chapter III]{HilleP57}.

\end{proof}

\bibliographystyle{alpha}
\bibliography{FPT}

\end{document}